\documentclass[12pt,a4paper]{amsart}

\usepackage{amssymb}
\usepackage{esint}
\usepackage[numeric, abbrev, nobysame]{amsrefs}
\usepackage{amscd}

\def\frak{\mathfrak}
\def\Bbb{\mathbb}
\def\Cal{\mathcal}

\let\phi\varphi

\newcommand{\x}{\times}
\renewcommand{\o}{\circ}

\newcommand{\al}{\alpha}
\newcommand{\be}{\beta}

\newcommand{\ka}{\kappa}

\newcommand{\om}{\omega}
\newcommand{\ph}{\phi}

\newcommand{\si}{\sigma}

\newcommand{\Ga}{\Gamma}
\newcommand{\La}{\Lambda}
\newcommand{\Ph}{\Phi}

\newcommand{\Om}{\Omega}

\newcommand{\vol}{\operatorname{vol}}
\newcommand{\im}{\operatorname{im}}

\newcommand{\id}{\operatorname{id}}

\newcommand{\ad}{\operatorname{ad}}

\newcommand{\tr}{\operatorname{tr}}

\newcounter{theorem}
\numberwithin{theorem}{section}
\numberwithin{equation}{section}

\newtheorem{thm}[theorem]{Theorem}
\newtheorem*{thm*}{Theorem \thesubsection}
\newtheorem{lemma}[theorem]{Lemma}
\newtheorem{prop}[theorem]{Proposition}
\newtheorem{cor}[theorem]{Corollary}
\newtheorem*{lemma*}{Lemma \thesubsection}
\newtheorem*{prop*}{Proposition \thesubsection}
\newtheorem*{cor*}{Corollary \thesubsection}

\theoremstyle{definition}
\newtheorem{definition}[theorem]{Definition}
\newtheorem*{definition*}{Definition \thesubsection}

\newtheorem*{example*}{Example \thesubsection}
\theoremstyle{remark}

\newtheorem*{remark*}{Remark \thesubsection}

\def\sideremark#1{\ifvmode\leavevmode\fi\vadjust{\vbox to0pt{\vss
 \hbox to 0pt{\hskip\hsize\hskip1em
 \vbox{\hsize3cm\tiny\raggedright\pretolerance10000
  \noindent #1\hfill}\hss}\vbox to8pt{\vfil}\vss}}}%
                        
\begin{document}
\renewcommand{\today}{} 

\title{A Poisson transform adapted to the Rumin complex}

\author{Andreas \v Cap, Christoph Harrach, and Pierre Julg}

\address{A.\v C. and C.H. : Faculty of Mathematics\\
  University of Vienna\\
  Oskar--Mor\-gen\-stern--Platz 1\\
  1090 Wien\\
  Austria\\
  P.J.: MAPMO UMR 7349, F\'{e}d\'{e}ration Denis Poisson \\
Universit\'{e} d'Orl\'{e}ans, Coll\'{e}gium Sciences et Techniques \\
B\^{a}timent de math\'{e}\-ma\-tiques - Route de Chartres \\
B.P. 6759, 45067 Orl\'{e}ans cedex 2 \\
France} 

\email{Andreas.Cap@univie.ac.at} 
\email{Christoph.Harrach@univie.ac.at}
\email{Pierre.Julg@univ-orleans.fr}

\date{April 23, 2020} 


\begin{abstract}
  Let $G$ be a semisimple Lie group with finite center, $K\subset G$ a maximal
  compact subgroup, and $P\subset G$ a parabolic subgroup. Following ideas of P.Y.\
  Gaillard, one may use $G$-invariant differential forms on $G/K\x G/P$ to construct
  $G$-equivariant Poisson transforms mapping differential forms on $G/P$ to
  differential forms on $G/K$. Such invariant forms can be constructed using finite
  dimensional representation theory. In this general setting, we first prove that the
  transforms that always produce harmonic forms are exactly those that descend from
  the de Rham complex on $G/P$ to the associated Bernstein-Gelfand-Gelfand (or BGG)
  complex in a well defined sense.
  
  The main part of the article is devoted to an explicit construction of such
  transforms with additional favorable properties in the case that
  $G=SU(n+1,1)$. Thus $G/P$ is $S^{2n+1}$ with its natural CR structure and the
  relevant BGG complex is the Rumin complex, while $G/K$ is complex hyperbolic space
  of complex dimension $n+1$. The construction is carried out both for complex and
  for real differential forms and the compatibility of the transforms with the
  natural operators that are available on their sources and targets are analyzed in
  detail.
\end{abstract}

\thanks{First and second author supported by project P27072--N25 of the Austrian
  Science Fund (FWF)}

\maketitle

\pagestyle{myheadings}\markboth{\v Cap, Harrach and Julg}{Poisson transforms adapted
  to the Rumin complex}

\section{Introduction}\label{1}
The hyperbolic spaces over $\Bbb K:=\Bbb R$, $\Bbb C$ or $\Bbb H$ together with the
boundary spheres at infinity provide the simplest examples of compactifications of
Riemannian symmetric spaces of the non--compact type. The realizations of the spheres
$S^n$, $S^{2n+1}$, and $S^{4n+3}$ as the boundary at infinity correspond to certain
geometric structures on the spheres. To describe these explicitly, let us denote by
$G$ the special unitary group of a non--degenerate Hermitian form on $\Bbb K^{n+2}$
of Lorentzian signature $(n+1,1)$. Thus $G$ is $SO(n+1,1)$, $SU(n+1,1)$ and
$Sp(n+1,1)$ for $\Bbb K=\Bbb R$, $\Bbb C$ and $\Bbb H$, respectively. Further, let
$K\subset G$ be the stabilizer of a fixed one--dimensional $\Bbb K$--subspace in
$\Bbb K^{n+2}$ on which the Hermitian form is negative definite. Then $K$ acts
trivially on the given line, so it can also be viewed as the stabilizer of a negative
vector, and it turns out that $K\subset G$ is a maximal compact subgroup.

Since $G$ acts transitively on the space of negative lines, viewed as a subspace in
$\Bbb KP^{n+1}$, we can view $G/K$ as the space of these negative lines. The boundary
of this space in $\Bbb KP^{n+1}$ can be identified with the space of isotropic lines,
so again this is a homogeneous space of $G$. The stabilizer $P\subset G$ of a fixed
isotropic line in $\Bbb K^{n+2}$ is well known to be the unique non--trivial
parabolic subgroup of $G$, so $G/P$ realizes the boundary sphere as a generalized
flag variety of $G$.

The spaces $G/P$ are the homogeneous models of so--called \textit{parabolic
  geometries} of type $(G,P)$, which are the geometric structures alluded to
above. According to the different choices of $\Bbb K$, this gives rise to the locally
flat conformal structure on $S^n$, the spherical CR structure on $S^{2n+1}$ and the
locally flat quaternionic contact structure on $S^{4n+3}$. These structures can be
realized as ``infinities'' of the hyperbolic metric, which is a Riemannian metric, a
K\"ahler metric, or a quaternion K\"ahler metric of constant curvature according to
the choice of $\Bbb K$. The construction of such geometries at infinity extends to
more general Einstein metrics of the types mentioned above, see \cite{Biquard}.

These boundary spheres can be used to construct eigenfunctions for the Laplace
operator on the corresponding Riemannian symmetric space. This is done by the
\textit{Poisson transform}, which is an integral operator assigning to every
continuous function on the boundary $G/P$ its average over the action of $K$,
c.f. \cite{Helgason_GASS}*{II.3.4}. It turns out that each eigenfunction of the
Laplace operator is obtained in this way and that its asymptotic behaviour towards
the boundary recollects the initial data. Due to its importance the Poisson transform
was generalized to map between sections of vector bundles in \cite{vanderVen} for the
case of symmetric spaces of real rank $1$ and independently in \cite{Yang} and
\cite{Olbrich} for arbitrary symmetric spaces of noncompact type.

A variant of the Poisson transform in the case of realy hyperbolic space
$\Cal H_{\Bbb R}^{n+1}$ was introduced by P.Y.\ Gaillard in
\cite{Gaillard:real}. Gaillard's construction does not only work on functions (or
densities) but on general differential forms. The basic idea of this construction is
to start with differential forms on the product $\Cal H_{\Bbb R}^{n+1}\x S^n$, for
which there is a natural notion of bidegree. Gaillard used geometric ideas going back
to Thurston to construct a $G$--invariant differential form $\pi_k$ of bidegree
$(k,n-k)$ on this product for each $k=0,\dots,n$. Given a $k$--form
$\al\in\Om^k(S^n)$, one can pull it back to the product, wedge it with $\pi_k$ to
obtain a form of bidegree $(k,n)$ and then integrate over the fibers to obtain a
$k$--form on $\Cal H_{\Bbb R}^{n+1}$, which is then defined to be the Poisson
transform of $\al$.

A simple computation shows that $G$--invariance of $\pi_k$ implies that the resulting
Poisson transform defines a $G$--equivariant map $\Om^k(S^n)\to\Om^k(\Cal H_{\Bbb
  R}^{n+1})$. The advantage of this construction is that one can directly relate the
transform to the exterior derivatives on both factors and to the codifferential and
the Laplacian on $\Om^*(\Cal H_{\Bbb R}^{n+1})$. In particular, the transform always
produces co--closed, harmonic forms on $\Cal H_{\Bbb R}^{n+1}$. Gaillard's
construction was partly generalized to the case of complex hyperbolic space in
\cite{Gaillard:complex}, but things get much more involved there and the results are
much less satisfactory.

In joint research of the first and third author, it was observed that there is a
purely algebraic approach to the construction of invariant forms on the product space
for each choice of $\Bbb K$. This is based on the rather simple observation that $G$
acts transitively on the product $G/K\x G/P$, so this can be realized as $G/M$, where
$M=K\cap P$. Consequently, describing $G$--invariant forms on the product reduces to
questions of finite dimensional representation theory of $M$. This not only refers to
the determination of invariant forms but also to the study of the relation between
the Poisson transforms determined by invariant forms on $G/M$ and natural operations
on differential forms on $G/P$ and $G/K$, respectively. This approach was worked out
in a general setting (involving arbitrary generalized flag varieties and allowing
also forms with values in certain vector bundles) in the PhD thesis of the second
author \cite{Harrach:Diss}, see also \cite{Harrach:Srni} and \cite{Harrach:BGG}.

Considering this general setting, it quickly turns out that the case of real
hyperbolic space as treated by Gaillard is deceptively simple. In this case, the
invariant forms used to define the Poisson transform are essentially unique up to
scale. Already in slightly more complicated cases, there is a large supply of
invariant forms and thus of possible Poisson transforms, and the problem rather is to
``design'' Poisson transforms with favorable properties. An important ingredient in
that direction is that apart from the exterior derivative, there are additional
invariant differential operators on differential forms on $G/P$. These can be used to
pass from the de Rham complex to the so-called Bernstein--Gelfand--Gelfand complex
(or BGG complex), which also computes the (twisted) de Rham cohomology. While the BGG
complex involves operators of order bigger than one, it has the advantage that all
the bundles in the complex are associated to irreducible representations of $P$ (and
thus of $M$). In contrast, in the (twisted) de Rham complex one typically meets
bundles that are induced by representations that are indecomposable but not
irreducible.

The CR--sphere provides the simplest example in which there is such a BGG refinement
of the standard real or complex de Rham complex. It turns out that in this case the
refinement does not really depend on the CR structure but only on the underlying
contact structure. A construction of this complex on arbitrary contact manifolds
(without using representation theory or touching the question of invariance) was
given by M.\ Rumin in \cite{Rumin} and therefore it is known as the Rumin complex.

The relation between the de Rham complex and the BGG complex is slightly intricate,
since the latter is a sub--quotient of the former. There is a calculus that can be
used to construct BGG complexes (and also their analogs for curved parabolic
geometries, see \cite{CSS-BGG}) which allows to identify conditions on a Poisson
transform which ensures that, while still acting on differential forms, it descends
to a transform defined on the BGG complex. It turns out that in the general setting,
this property is closely related to the fact that the values of the transform are
harmonic forms. The main purpose of this article is to explicitly construct such a
Poisson transform between the Rumin complex on the CR sphere and harmonic forms on
complex hyperbolic space and study some of its properties. The construction of this
transform does depend on the CR structure on the sphere (and not only on its contact
structure) and require a rather careful analysis of the invariant differential forms
available in this setting.

Poisson transforms defined on the Rumin complex play an important role in a long-term
project of the third author. This aims at using the first part of a BGG complex
together with carefully chosen Poisson transforms defined on the middle degree space
in order to get an index one complex whose class in Kasparov $KK$-theory is the
so-called $\gamma$ element. This is a crucial step towards the proof of the
Baum-Connes conjecture with coefficients in the case of simple Lie groups of rank
one. See \cite{Julg-Kasparov} for the case of $SU(n,1)$ and \cite{Julg} in
general. Apart from their intrinsic interest, some of the results in this article
should also be viewed as providing a basis for carrying out this program in the case
under consideration.

\section{Poisson transforms on differential forms}\label{2}

The first part of this section describes a general scheme for constructing Poisson
transforms. For a semisimple Lie group $G$, these transforms map differential forms
on a generalized flag variety $G/P$ to differential forms on the quotient $G/K$ by
the maximal compact subgroup, which is a Riemannian symmetric space of the
non-compact type. These transforms are induced by $G$-invariant differential forms on
the homogeneous space $G/M$, where $M=K\cap P$, that admit a description in terms of
finite dimensional representation theory. In this general setting, we describe a
characterization of those transforms whose images consist of harmonic forms on $G/K$,
which provides a link to the machinery of BGG sequences on $G/P$. 

In the later parts of the article, we specialize to the case that $G=SU(n+1,1)$ and
$P$ is the unique parabolic subgroup of $G$.  We give elementary explicit
descriptions of the structures needed to study Poisson transforms in that case later.

\subsection{The general setup and Poisson transforms}\label{2.0}\label{3.1}
In this general part we keep things rather short and abstract, they will be made
explicit in the special case that $G=SU(n+1,1)$ soon. The basic idea to define
Poisson transforms via invariant forms on a product goes back to
\cite{Gaillard:real}, the general version was studied in
\cite{Harrach:Diss}. Consider a non-compact semisimple Lie group $G$ with finite
center, let $P\subset G$ be a parabolic subgroup and let $K\subset G$ be the maximal
compact subgroup. Then $G/P$ is a generalized flag variety of $G$ and it is well
known that the restriction of the natural $G$-action on $G/P$ to the subgroup $K$ is
transitive. Thus $G/P$ can be identified with $K/M$, where $M:=K\cap P$, which in
particular shows that $G/P$ is compact. On the other hand, $G/K$ is a Riemannian
symmetric space of the non-compact type, and transitivity of the $K$-action on $G/P$
shows that $G$ acts transitively on $G/K\x G/P$. Thus, the product can be identified
with the homogeneous space $G/M$.

Let us denote the resulting projections from $G/M$ onto the two factors by $\pi_K$
and $\pi_P$, respectively. Correspondingly, the tangent bundle of $G/M$ decomposes as
$T'\oplus T''$ where $T'=\pi_K^*T(G/K)$ and $T''=\pi_P^*T(G/P)$, implying that there
is a well defined notion of bidegree for differential forms on $G/M$. Explicitly, a
$k$--form $\al$ on $G/M$ is of bidegree $(i,j)$ with $i+j=k$ if for entries, which
are either from $T'$ or from $T''$, it vanishes unless there are exactly $i$ entries
from $T'$ and $j$ entries from $T''$. We will indicate the resulting decomposition of
forms as $\Om^k(G/M)=\oplus_{i+j=k}\Om^{(i,j)}(G/M)$.

In particular, for $\al\in\Om^k(G/P)$, the pullback $\pi_P^*\al$ lies in
$\Om^{(0,k)}(G/M)$. Given a form $\phi\in\Om^{(\ell,N-k)}(G/M)$, where $N=\dim(G/P)$,
we can form the wedge product $\phi\wedge\pi_P^*\al\in\Om^{(\ell,N)}(G/M)$. This form
can be integrated over the fibers of the projection $G/M\to G/K$ to define an
$\ell$-form on $G/K$. Thus we obtain an integral operator
\begin{equation}\label{eq:Poiss-def}
 \Phi: \Om^k(G/P) \to \Om^\ell(G/K) \qquad\qquad \al\mapsto \fint_{G/P} \phi \wedge
 \pi_P^*\alpha. 
\end{equation}
It is easy to show that $\Phi$ is $G$--equivariant if and only if the form $\phi$ is
$G$--invariant.

\begin{definition}\label{def3.1}
  Let $\phi\in\Om^{(\ell,N-k)}(G/M)$ be a $G$--invariant differential form. Then the
  corresponding $G$--equivariant operator $\Phi:\Om^k(G/P) \to \Om^{\ell}(G/K)$ from
  \eqref{eq:Poiss-def} is called a \emph{Poisson transform} and $\phi$ is called its
  \emph{Poisson kernel}.
\end{definition}

Therefore, a Poisson transform in the sense of this definition is characterized by
its $G$-invariant Poisson kernel $\phi$. In turn, by Theorem 1.4.4 of \cite{book},
$\phi$ is fully determined by its value $\phi(eM)$ at the origin $eM\in G/M$, which is
an $M$--invariant element in the corresponding finite dimensional $M$--representation
$\La^*(\frak g/\frak m)^*$.

The basic advantage of working with differential forms is that it makes several
natural operations available, which of course include the exterior derivatives on the
source and the target of the transform. The advantage of the construction via forms on
$G/M$ is that such operations can be described nicely in terms of operations on
Poisson kernels, which we will do for $G=SU(n+1,1)$ below. On the Riemannian
symmetric space $G/K$, we also have the Hodge star operator, the codifferential and
the Laplace Beltrami operator available. On the generalized flag variety $G/P$ there
are some less well know natural operations that we will discuss next.

\subsection{The codifferential and BGG sequences}\label{2.0a}
It is well known that $T(G/P)$ is the homogeneous vector bundle $G\x_P(\frak g/\frak
p)$. Now the Killing form induces a $P$-equivariant duality between $\frak g/\frak p$
and the nilradical $\frak p_+$ of $\frak p$. Thus, the bundles $\La^kT^*(G/P)$ of
differential forms are the homogeneous bundles associated to the representations
$\La^k\frak p_+$ of $P$, which are the chain spaces of the Lie algebra $\frak p_+$
with coefficients in the trivial representation.

On these chain spaces, there is a natural Lie algebra homology differential. In the
theory of parabolic geometries, this is traditionally denoted by $\partial^*$ and
called the \textit{Kostant codifferential} and we will stick to this
tradition. Explicitly, $\partial^*:\La^k\frak p_+\to\La^{k-1}\frak p_+$ is given by
\begin{equation}\label{eq:def-codiff}
\partial^*(Z_1 \wedge \dots \wedge Z_k) = \textstyle\sum_{i<j}
(-1)^{i+j} [Z_i, Z_j] \wedge Z_1 \wedge \dots \wedge \hat{Z_i} \wedge
\dots \wedge \hat{Z_j} \wedge \dots \wedge Z_k,
\end{equation}
where the $Z_\ell$ are in $\frak p_+$ and hats denote omission. From the explicit
formula it is immediate that $\partial^*$ is $P$--equivariant, so its kernel and its
image are $P$--invariant subspaces in $\La^k\frak p_+$. Moreover, the homology spaces
$H_k(\frak p_+) := \ker(\partial^*)/\im(\partial^*)$ naturally are $P$--modules.

The $P$--homomorphisms $\partial^*$ induce $G$--equivariant bundle maps
$\Lambda^kT^*(G/P) \to \Lambda^{k-1}T^*(G/P)$. We use the symbol $\partial^*$ also
for these bundle maps and for the induced tensorial operators on differential
forms. The kernels and images of these bundle maps induce $G$-invariant subbundles of
each $\Lambda^kT^*(G/P)$. In particular, we get $\im(\partial^*) \subset
\ker(\partial^*)\subset\La^kT^*(G/P)$ and we denote by $\Cal H_k$ their quotient
bundle, which by construction is associated to the $P$-representation $H_k(\frak
p_+)$.

Let $\pi_H:\Ga(\ker(\partial^*)) \to \Ga(\Cal H_k)$ be the tensorial projection
induced by the quotient projection. The machinery of BGG sequences introduced in
\cite{CSS-BGG} and \cite{Calderbank-Diemer} is based on the construction of a natural
differential operator $L:\Ga(\Cal H_k)\to \Ga(\ker(\partial^*))$ which splits the
projection $\pi_H$. Viewed as an operator to $\Ga(\Cal H_k)\to\Om^k(G/P)$, $L$ is
characterized by this splitting property (i.e.\ $\partial^*\o L=0$ and
$\pi_H\o L=\id$) and the single condition that $\partial^*\o d\o L=0$, where $d$
denotes the exterior derivative.

For our purpose, the nicest description uses the operator $\square^R := \partial^*\o
d + d\o \partial^*$, which defines an endomorphism of $\Om^k(G/P)$ for each $k$, see
Section 3 of \cite{Rel-BGG2}. It turns out that
$\ker(\square^R)\subset\Ga(\ker(\partial^*))\subset\Om^k(G/P)$ for each $k$ and that
$\pi_H$ restricts to a linear isomorphism from this kernel onto $\Ga(\Cal H_k)$ whose
inverse is precisely the splitting operator. Moreover, for each $k$, the inverse of
this isomorphism can also be realized by applying an operator that can be written as
a universal polynomial in $\square^R$ to any section of $\ker(\partial^*)$
representing the given section of $\Cal H_k$. Having $L$, one defines the $k$th BGG
operator $D_k:\Ga(\Cal H_k)\to\Ga(\Cal H_{k+1})$ as $D_k := \pi_H \o d \o L$ and by
construction this is a $G$--equivariant differential operator. From the fact that
$d^2=0$, one easily concludes that these operators form a differential complex,
i.e. satisfy $D_{k+1} \o D_k = 0$, which is called the \emph{BGG complex} (associated
to the trivial representation).

\subsection{The action of the Casimir element}\label{2.4}
The operators $\square^R$ we have met in the BGG construction have immediate
relevance for the study of $G$--equivariant maps defined on $\Om^*(G/P)$. To explain
this, let us recall that the Casimir element of the semisimple Lie algebra $\frak g$
induces a differential operator on the space of sections of any homogeneous vector
bundle $E\to G/H$, where $G$ is any Lie group with Lie algebra $\frak g$ and
$H\subset G$ is any closed subgroup. Such a vector bundle is induced by a (finite
dimensional) representation $\Bbb V$ of $H$, in the sense that $E=G\x_H\Bbb V$. There
is a natural action of $G$ on $\Ga(E)$ by linear maps, which in the equivalent
picture of $H$--equivariant smooth functions $G\to \Bbb V$ is given by $(g\cdot
f)(g')=f(g^{-1}g')$. This shows that there is an induced action of the Lie algebra
$\frak g$, for which $X\in\frak g$ acts on $f$ as differentiation by the
\textit{right} invariant vector field $R_X\in\frak X(G)$ generated by $X$.

This action naturally extends to an action of the universal enveloping algebra
$\Cal U(\frak g)$. In particular, the Casimir element $\Cal C$ acts by a differential
operator $D_{\Cal C}:\Ga(E)\to\Ga(E)$, which is $G$--equivariant by invariance of
$\Cal C$, compare with Section 2.4 of \cite{CurvedCas}. Returning to our setting, we
conclude in particular that any $G$--equivariant linear map
$\Phi:\Om^k(G/P) \to \Om^{\ell}(G/K)$ that satisfies a weak continuity assumption
has to be compatible with the actions of the two operators $D_{\Cal C}$. Since the
precise nature of these continuity conditions is not important, we state the result
under a weak assumption, namely that $\Phi$ is bounded for the natural bornologies on
the spaces in question, see the book \cite{Kriegl-Michor}. This condition on a linear
map ensures that it is smooth in the sense of mapping smooth curves to smooth curves.

\begin{prop}\label{prop2.4}
  Let $\Phi:\Om^k(G/P) \to \Om^{\ell}(G/K)$ be a $G$--equivariant bounded linear
  operator. Then the following are equivalent:
\begin{itemize}
 \item[$(i)$] For all $\al \in \Om^k(G/P)$ the differential form $\Phi(\alpha)$ is harmonic.
 \item[$(ii)$] For all $\al \in \Om^k(G/P)$ we have $\Phi(\Box^R\alpha) = 0$.
 \item[$(iii)$] We have $\Phi \circ \partial^* = 0$ and $\Phi \circ d \circ
   \partial^* = 0$.
\end{itemize}
\end{prop}
\begin{proof}
  As we have note above, $\Phi$ is smooth in the sense that it maps smooth curves to
  smooth curves. Applying this to the action of a one-parameter subgroup in $G$ on a
  fixed form and differentiating, one concludes that $\Phi$ is equivariant for the
  infinitesimal actions of $\frak g$. The definition of the Casimir then directly
  implies that $\Phi$ intertwines the actions of the Casimir operator on the two
  spaces.

  Since the canonical Riemannian metric on $G/K$ is induced by the Killing form on
  $\frak g$, the Casimir operator on $\Om^*(G/K)$ is a multiple of the Laplace
  Beltrami operator $\Delta$, see e.g.\ p.\ 385 in \cite{Matsushima-Murakami}. On the
  other hand, it was shown in Corollary 1 of \cite{CurvedCas} that
  $D_{\Cal C}=2\Box^R$ on $\Om^*(G/P)$. This readily implies the equivalence of $(i)$
  and $(ii)$, and the definition of $\square^R$ shows that $(iii)$ implies
  $(ii)$.

  Thus it remains to show that $(ii)$ implies $(iii)$. To do this, we observe that
  Theorem 5.2 of \cite{Calderbank-Diemer} or Theorem 3.9 of \cite{Rel-BGG2} shows
  that $\square^R$ is invertible on $\Ga(\im(\partial^*))$. Thus we can write
$$
 \Phi \circ \partial^* = \Phi \circ \Box^R \circ (\Box^R)^{-1} \circ \partial^* 
 $$
 and this vanishes by $(ii)$. Knowing that $\Phi\o\partial^*=0$, the definition of
 $\square^R$ readily implies that
 $\Phi \circ d \circ \partial^* = \Phi \circ \Box^R = 0$.
\end{proof}

From this we can prove that an intertwining operator with harmonic values
automatically descends to the BGG complex:

\begin{cor}\label{cor2.4}
  Let $\Phi:\Om^k(G/P)\to\Om^{\ell}(G/K)$ be an bounded linear intertwining operator
  which satisfies the equivalent conditions from Proposition \ref{prop2.4}, and let
  $\Cal H_k\to G/P$ be the $k$th homology bundle.

  Then $\Phi$ descends to a well defined, $G$--equivariant map
  $$
  \underline{\Phi}:\Ga(\Cal H_k) \to \Om^{\ell}(G/K),
  $$ whose image is contained in the space of harmonic differential forms on
  $G/K$. Explicitly, for $\si\in\Ga(\Cal H_k)$, $\underline{\Phi}(\si)$ can be
  computed as $\Phi(\al)$ for any $\al\in\Ga(\ker(\partial^*))$ such that
  $\pi_H\o\al=\si$. Moreover, for $\tau\in\Ga(\Cal H_{k-1})$, we get
  $\underline{\Phi}(D_{k-1}(\tau))=\Phi(d\be)$ for any $\be$ such that
  $\pi_H\o\be=\tau$.
\end{cor}
\begin{proof}
From Proposition \ref{prop2.4}, we know that $\Phi\o\partial^*=0$, so the
restriction of $\Phi$ to $\Ga(\ker(\partial^*))$ descends to an operator
$\underline{\Phi}$, which has the first claimed property.

Take $\tau\in\Ga(\Cal H_{k-1})$ and consider $dL(\tau)$, where $L$ is the splitting
operator. This is a section of $\ker(\partial^*)$ and by definition $\pi_H\o
dL(\tau)=D_{k-1}(\tau)$. Thus we see that
$\underline{\Phi}(D_{k-1}(\tau))=\Phi(dL(\tau))$. For $\be\in\Ga(\ker(\partial^*))$
such that $\pi_H\o\be=\tau$, we get $L(\tau)-\be\in\Ga(\im(\partial^*))$ and in the
proof of Proposition \ref{prop2.4}, we have noted that this implies that it lies in
the image of $\square^R$. By definition, $d^2=0$ implies that $\square^R$ commutes
with $d$, so $dL(\tau)-d\be$ lies in the image of $\square^R$, and thus
$\Phi(dL(\tau))=\Phi(d\be)$.
\end{proof}

\section{The case of complex hyperbolic space}\label{3}
We next describe the machinery discussed in Chapter \ref{2} in an elementary and
explicit way in the case that $G=SU(n+1,1)$ and $P$ is the unique parabolic subgroup
of $G$. Thus $G/K$ is complex hyperbolic space and $G/P$ can be identified with the
boundary sphere at infinity, which inherits a natural CR structure.

\subsection{Complex hyperbolic space}\label{2.1}
From this point on, we will restrict our attention to the case that $G=SU(n+1,1)$ and
$P$ is the unique parabolic subgroup of $G$. Fixing a Lorentzian Hermitian form $h$
on $V:=\Bbb C^{n+2}$, and realizing $G$ as $SU(h)$, we can realize $K$ as the
stabilizer of a complex line $\ell_-\subset V$ on which $h$ is negative
definite. Then $K$ preserves the orthocomplement $\ell_-^\perp$ and acts unitarily on
both $\ell_-^\perp$ and on $\ell_-$, which shows that $K\cong S(U(n+1)\times U(1))$.
This readily implies that $\frak g/\frak k$ is isomorphic to the space $L(\ell_-,\ell_-^\perp)$ of linear maps from $\ell_-$ to $\ell_-^\perp$ endowed with
the natural action of $K$. The space $G/K$ can then be identified with the space of
all complex lines in $V$ on which $h$ is negative definite, so this is complex
hyperbolic space of dimension $n+1$.

The complex structure on $L(\ell_-,\ell_-^\perp)$ clearly is $K$--invariant and there
is a positive definite, Hermitian inner product on this space defined by
$(X,Y)\mapsto \tr(X^*\o Y)$, which is $K$--invariant, too. Hence the homogeneous
space $G/K$ carries a $G$--invariant almost complex structure $J$ and a
$G$--invariant Hermitian metric. It is well known that these data actually make $G/K$
into a complete K\"ahler manifold. In particular, we obtain the usual decomposition
of complex valued differential forms into $(p,q)$--types, which we indicate by
$\Om^*(G/K,\Bbb C)=\oplus_{0\leq p,q\leq n+1}\Om^{p,q}(G/K)$.

Using the (real valued) $G$--invariant Riemannian metric $g$, we obtain the
``musical'' operator ${}^\flat: T(G/K) \to T^*(G/K)$ via
$\xi^{\flat}(\eta) = g(\xi, \eta)$ for all $\xi$, $\eta \in T(G/K)$, and we denote
the induced map on the level of sections by the same symbol. The K\"ahler form
$\omega \in \Omega^2(G/K)$ associated to $g$ is the $G$--invariant differential form
characterized by $\omega(\xi, \eta) = g(J\xi, \eta)$ for all vector fields $\xi$,
$\eta$ on $G/K$. In particular, the exterior power $\frac{1}{(n+1)!}\omega^{n+1}$ is
the volume form $vol$ on $G/K$, which by construction is also $G$--invariant. We will
denote the complex extensions of the K\"{a}hler form and the volume form by the same
symbols. Note that these lie in $\Omega^{1,1}(G/K)$ and $\Omega^{n+1,n+1}(G/K)$,
respectively. The Hermitian extension of $g$ can be used to identify the holomorphic
part $T^{1,0}(G/K)$ in the complexified tangent bundle of $G/K$ with $T^*_{0,1}(G/K)$
and then in turn induces a complex bilinear dual pairing
$T^*_{1,0}(G/K)\x T^*_{0,1}(G/K)\to\Bbb C$ that extends to exterior powers in the
usual way. Denoting this by $\langle\ ,\ \rangle$, we get the
Hodge--$*$--operator. This is the $G$-equivariant complex linear map
$\ast: \La^{p,q}T^*(G/K) \to \La^{n+1-q, n+1-p}T^*(G/K)$ characterized by
$$
 \alpha \wedge \ast\beta =  \langle \alpha, \beta \rangle \vol
 $$
 for all $\alpha\in \La^{q,p}T^*(G/K)$, $\be \in \La^{p,q}T^*(G/K)$. We also denote
 the induced operator on forms by $\ast$. Denoting by $d$ the exterior derivative on
 complex valued forms, we then get the codifferential $\delta := - \ast d \ast$ and
 the Laplace-Beltrami operator $\Delta := d\delta + \delta d$.

The decomposition of the complexified tangent bundle into holomorphic and
antiholomorphic vectors induces a splitting of the exterior derivative on $G/K$ into
the sum $d = \partial + \overline{\partial}$, where the first operator maps a form of
type $(p,q)$ to a form of type $(p+1,q)$ and the second is defined by
$\overline{\partial}(\alpha) := \overline{(\partial\overline{\alpha})}$. We say that
a differential form $\alpha \in \Omega^k(G/K, \Bbb C)$ is \emph{holomorphic}
(respectively, \emph{antiholomorphic}) if $\overline{\partial}\alpha = 0$
(respectively $\partial\alpha = 0$). Similarly, we can decompose the codifferential
$\delta = \partial^* + \overline{\partial^*}$ with $\partial^* = - \ast
\overline{\partial} \ast$ and $\overline{\partial}^* = - \ast \partial \ast$.

Finally, the wedge product with the K\"{a}hler form $\omega \in \Omega^{1,1}(G/K)$
defines the \emph{Lefschetz map}
$\Cal L:\Omega^{p,q}(G/K) \to \Omega^{p+1,q+1}(G/K)$, which is $G$-equivariant by
construction. Its adjoint with respect to the Riemannian metric $g$ is the
$G$-equivariant map $\Cal L^*: \Omega^{p,q}(G/K) \to \Omega^{p-1,q-1}(G/K)$, which
can be computed as $\Cal L^* = - \ast \o\Cal L\o \ast$. We say that
$\alpha \in \Omega^{p,q}(G/K)$ is \emph{primitive} iff $\Cal L^*\alpha = 0$ and
\emph{coprimitive} iff $\Cal L\alpha=0$. It is well known (c.f. Theorem 3.11 (b) in
\cite{Wells}) that the degree of a primitive differential form is at most half of the
real dimension of $G/K$, while coprimitive forms exist only above that degree.

\subsection{The CR sphere}\label{2.2}
In the notation introduced above, we define $P\subset G$ to be the stabilizer of a
fixed complex line $\ell_0\subset V$, which is isotropic for $h$. It is well known
that $P$ is a minimal parabolic subgroup of $G$ and the space $G/P$ of complex
isotropic lines in $V$ is diffeomorphic to the sphere $S^{2n+1}$. Such a
diffeomorphism can be obtained by fixing a vector $v_-\in\ell_-$ such that
$h(v_-,v_-)=-1$ and then sending each unit vector $v\in\ell_-^\perp$ to the line
spanned by $v+v_-$, which is visibly isotropic. This also shows how $G/P$ can be
viewed as the boundary at infinity of $G/K$, since mapping $v$ to the complex line
spanned by $v+v_-$ also restricts to a diffeomorphism from the open unit ball in
$\ell_-^\perp$ onto the space of negative lines in $V$, which can be identified with
$G/K$.

Since $\ell_0$ is isotropic, the orthogonal space $\ell_0^\perp$ is a complex
hyperplane in $V$ which contains $\ell_0$, so we can view this as defining a
filtration $\ell_0\subset\ell_0^\perp\subset V$ of $V$ by complex subspaces. Let us
write this filtration as $V^1\subset V^0\subset V^{-1}$ and define $V^j=\{0\}$ for
$j>1$ and $V^j=V$ for $j<-1$. Then we get an induced filtration of the Lie algebra
$\frak g$ of $G$ by defining $\frak g^i=\{X\in\frak g:\forall j:X\cdot V^j\subset
V^{i+j}\}$. By definition, this filtration is compatible with the Lie bracket in the
sense that $[\frak g^i,\frak g^j]\subset\frak g^{i+j}$ for all $i,j$.

Since $P\subset G$ can be characterized as the stabilizer of the filtration of $V$,
we see that $\frak g^0=\frak p$ and that each $\frak g^i$ is a $P$--invariant
subspace in $\frak g$.  From the definition, it is evident that $\frak g^i=\frak g$
for $i\leq -2$ and $\frak g^i=0$ for $i>2$, but for indices between $-2$ and $2$, we
get a proper filtration. In particular, $\frak g^{-1}/\frak p$ is a $P$--invariant
subspace in $\frak g/\frak p$, which is easily seen to be isomorphic to
$L(\ell_0,\ell_0^\perp/\ell_0)\cong\mathbb C^n$. On the other hand, $\frak g/\frak
g^{-1}$ has real dimension $1$ and the Lie bracket on $\frak g$ induces a bilinear,
skew symmetric map $ \frak g^{-1}/\frak p\x \frak g^{-1}/\frak p\to\frak g/\frak
g^{-1}$. This is easily seen to be the imaginary part of a positive definite
Hermitian form, compare with Section 4.2.4 of \cite{book}.

Since $T(G/P)=G\x_P(\frak g/\frak p)$, we conclude that $G\x_P(\frak g^{-1}/\frak p)$
defines a $G$--invariant corank-one subbundle $H\subset T(G/P)$ which carries a
$G$--invariant complex structure. The considerations about the Lie bracket show that
this is a contact structure, which makes $G/P$ into a strictly pseudoconvex partially
integrable almost CR structure of hypersurface type. It is easy to see that this is
indeed the standard (spherical) CR structure on $S^{2n+1}$ coming from the
realization as the unit sphere in $\Bbb C^{n+1}$.

It is well known that the filtration on $\frak g$ is actually induced by a grading,
which is not $P$--invariant, however. This can either be realized by the choice of a
complex isotropic line $\tilde\ell_0\subset V$, which is transverse to
$\ell_0^\perp$. Calling this line $V_{-1}$ and putting
$V_0:=\tilde\ell_0^\perp\cap\ell_0^\perp$ and $V_1=V^1$, we get $V=V_{-1}\oplus
V_0\oplus V_1$ such that $V^i=\oplus_{j\geq i}V_j$ for all $i$. Similarly as above,
this induces a grading $\frak g=\frak g_{-2}\oplus\dots\oplus\frak g_2$ compatible
with the Lie bracket that induces the given filtration in the sense that $\frak
g^i=\oplus_{j\geq i}\frak g_j$.

Alternatively, the grading can be obtained by choosing a Cartan subalgebra $\frak
h\subset\frak g$, which is contained in $\frak g^0$. Having made such a choice,
$\frak g^2:=\frak g_2$ becomes the highest root space, and one defines $\frak g_{-2}$
to be the lowest root space. Then there is a unique element $E\in [\frak g_{-2},\frak
  g_2]\subset\frak h$, which fits into a standard $\frak{sl}_2$--triple, and the
grading of $\frak g$ is the decomposition into eigenspaces for $\ad(E)$. A crucial
fact for our purposes is that the grading of $\frak g$ \textit{is} invariant under
the adjoint action of $M:=K\cap P$.

\subsection{The Rumin complex on the CR--sphere}\label{2.3}
In the special case we consider, the BGG complex as discussed in \ref{2.0a} turns out
to be a complex that is naturally defined on any contact manifold. This was first
constructed in this general setting by M.~Rumin, whence it is called the Rumin
complex. Let us briefly discuss this direct construction and its relation to the
general BGG construction. The construction is completely parallel for real and
complex valued forms, and we do not distinguish between the two cases here. 

Let $H\subset T(G/P)$ be the contact subbundle from above and define
$Q:=T(G/P)/H$. Then the short exact sequence $0\to H\to T(G/P)\to Q\to 0$ of
homogeneous vector bundles dualizes to a short exact sequence
$0\to Q^*\to T^*(G/P)\to H^*\to 0$. Since $Q^*$ has rank $1$, there is an induced
short exact sequence for the exterior powers of order $k=1,\dots,2n$, which has the
form
\begin{equation}\label{eq:cotangent}
  0\to \La^{k-1}H^*\otimes Q^*\to \La^kT^*(G/P)\to\La^kH^*\to 0 . 
\end{equation}
The fact that the Lie bracket induces a tensorial map $H\x H\to Q$ has a counterpart
in the dual picture of the exterior derivative. Take a $k$--form, which is a section
of the subbundle $\La^{k-1}H^*\otimes Q^*$, apply the exterior derivative and project
to a section of $\La^{k+1}H^*$ (i.e.~restrict the form to entries from $H$). The
result of this operation is easily seen to be linear over smooth functions, thus
defining a natural vector bundle map $\La^{k-1}H^*\otimes Q^*\to \La^{k+1}H^*$, which
is easily seen to be injective for $k\leq n$ and surjective for $k\geq n$.

This basically shows that it should be possible to pass to a ``part'' of the de Rham
complex without changing the cohomology. Indeed, it should be possible to ``leave
out'' a complement to the kernel of this bundle map in $\La^{k-1}H^*\otimes Q^*$ as
well as its image without changing the cohomology, since these are just mapped
isomorphically to each other by the exterior derivative. In M.\ Rumin's original
construction \cite{Rumin}, naturality was not an issue and he proceeded by choosing
splittings of the sequences \eqref{eq:cotangent} and then factoring by the irrelevant
parts. With a bit more effort, one can use the exact sequences \eqref{eq:cotangent}
and spectral sequence arguments (which can be made explicit as diagram chases in this
simple case) to obtain a construction of a complex which is manifestly invariant
under contactomorphisms, see \cite{BEGN} or \cite{Cap-Salac}.

To get to the setting of Section \ref{2.0a}, we observe that the nilradical
$\frak p_+$ of $\frak p$ in this simple case is the filtration component $\frak
g^1$.
(The duality between $\frak g/\frak p$ and $\frak g^1$ easily follows from
compatibility of the Killing form $B$ with the filtration and the fact that
$\frak p=\frak g^0$.)  As discussed in Section \ref{2.0a}, we have the homology space
$H_k(\frak p_+,\Bbb K)$ for $\Bbb K=\Bbb R$ or $\Bbb C$ and the corresponding bundles
$\Cal H_k$. We can next describe these homology bundles explicitly.

\begin{prop}\label{prop_homology_bundles}
 In terms of the exact sequences \eqref{eq:cotangent}, the homology bundle $\Cal H_k$
 satisfy the following. For $k\leq n$, $\Cal H_k$ is a subbundle of $\Lambda^k H^*$,
 while for $k \geq n+1$, $\Cal H_k$ is a quotient bundle of $\Lambda^{k-1} H^* \otimes
 Q^*$.
\end{prop}
\begin{proof}
  From the definition of $\partial^*$ from \eqref{eq:def-codiff} and the simple structure
  of $\frak p_+=\frak g_1\oplus\frak g_2$ in our case, we see that the bundle map
  $\partial^*:\La^kT^*(G/P)\to\La^{k-1}T^*(G/P)$ vanishes on the subbundle
  $\La^{k-1}H^*\otimes Q^*$ and has values in $\La^{k-2}H^*\otimes Q^*$. This it
  actually defines a bundle map
  $\underline{\partial^*}:\La^k H^* \to \La^{k-2} H^* \otimes Q^*$. It is a well
  known result that this map is surjective for all $k \le n+1$ and injective for
  $k \ge n+1$.

Thus we see that for $k \le n$, $\Cal H_k$ simply is
$\ker(\underline{\partial}^*)\subset \La^k H^*$. For $k \ge n+1$, the kernel of
$\partial^*$ equals $\La^{k-1} H^* \otimes Q^*$, so in these cases $\Cal H_k$ is the
quotient of $\La^{k-1} H^* \otimes Q^*$ by $\im(\underline{\partial^*})$, which
completes the proof.
\end{proof}

As discussed in Section \ref{2.0a}, we then obtain the BGG operators
$D_k:\Ga(\Cal H_k)\to\Ga(\Cal H_{k+1})$, which are $G$-equivariant differential
operators that form the BGG complex. It was shown in \cite{BEGN} that this complex
coincides with the Rumin complex on the CR sphere.

\subsection{Poisson transforms and natural operations}\label{3.2}
We are now ready to invoke the machinery discussed in Section \ref{2.0} in our
special case. As discussed there, $G$-invariant forms on $G/M$ or, equivalently,
$M$-invariant elements in $\La^*(\frak g/\frak m)^*$ give rise to Poisson
transforms mapping differential forms on $G/P$ to differential forms on $G/K$. We
next describe the composition of such transforms with natural operations in terms of
operations on the Poisson kernels.

First, observe that the decomposition of forms on $G/M$ according to bidegree induces
a splitting of the exterior derivative as $d = d_K + d_P$, where $d_K$ and $d_P$ map
forms of bidegree $(i,j)$ to forms of bidegree $(i+1,j)$ and $(i,j+1)$,
respectively. We call these operators, which are $G$--equivariant by construction,
the \emph{$K$-derivative} and the \emph{$P$-derivative}. Since $d^2=0$, we
immediately conclude that $d_K^2 = 0$, $d_P^2 = 0$ and $d_Kd_P = -d_Pd_K$. In view of
the complex structure on $G/K$, we get an obvious splitting
$d_K=\partial_K\oplus\overline{\partial}_K$.

Next, the Hodge--star operator on $G/K$ is induced by $K$--equivariant isomorphisms
$\La^k(\frak g/\frak k)^*\to\La^{2n+2-k}(\frak g/\frak k)^*$. Of course, these
isomorphisms are $M$--equi\-vari\-ant and tensorizing with appropriate identity maps
and passing to the induced tensorial operator, we obtain tensorial maps
$$
\ast_K:\Omega^{(i,j)}(G/M) \to \Omega^{(2n+2-i,j)}(G/M)
$$
for all $i$ and $j$, which we call the \emph{$K$--Hodge--star}. Observe that since
$G/K$ has even dimension, the inverse of $\ast_K$ is $-\ast_K$. Thus we define the
\emph{$K$--codifferential} and the \emph{$K$--Laplacian} on $\Om^*(G/M)$ by $\delta_K
:= - \ast_Kd_K\ast_K$ and $\Delta_K := d_K\delta_K + \delta_Kd_K$, respectively.

Let $\omega_K\in\Om^{(2,0)}(G/M)$ be the pullback of the K\"{a}hler form of $G/K$,
which of course is $G$--invariant. Using this, we define the \emph{$K$--Lefschetz
  map}
$$
\Cal L_K:\Om^{(i,j)}(G/M) \to \Omega^{(i+2,j)}(G/M)
$$
as the wedge product with $\omega_K$ and we consider its adjoint $\Cal L_K^*:= -
\ast_K\o \Cal L_K \o\ast_K$, which maps forms of bidegree $(i,j)$ to forms of
bidegree $(i-2,j)$.

Finally, the Kostant codifferential
$\partial^*:\La^j(\frak g/\frak p)^*\to\La^{j-1}(\frak g/\frak p)^*$ is
$P$--equivariant, and thus as above, this gives rise to a tensorial operation
$$
\partial^*_P:\Omega^{(i,j)}(G/M) \to \Omega^{(i,j-1)}(G/M),
$$
which we call the \emph{$P$-codifferential}. Here we fix the convention that for
$\al\in\Om^{(i,0)}(G/M)$ and $\be\in\Om^{(0,j)}(G/M)$, we put
$\partial^*_P(\al\wedge\be)=(-1)^i\al\wedge(\partial^*_P\be)$, which by linearity
defines the action on all of $\Om^{(i,j)}(G/M)$.

To analyze the relation of the Kostant codifferential with Poisson transforms, we
have to establish its compatibility with the wedge product on $\La^k\frak p_+$.
\begin{lemma}\label{lem2.3}
For each $k=1,\dots,2n$, $\al\in\La^k\frak p_+$ and
$\be\in\La^{2n+2-k}\frak p_+$ we get
$(\partial^*\al)\wedge\be=(-1)^k\al\wedge\partial^*\be$.
\end{lemma}
\begin{proof}
From the grading property it readily follows that the Killing form $B$ vanishes on
$\frak g_i\x\frak g_j$ unless $i+j=0$, so $B$ has to restrict to a non--degenerate
pairing on $\frak g_i\x\frak g_{-i}$ for each $i=0,1,2$. Choose elements
$\nu_{\pm}\in\frak g_{\pm2}$ such that $B(\nu_+,\nu_-)=1$ as well as bases
$\{\xi_s\}$ for $\frak g_{-1}$ and $\{\eta_s\}$ for $\frak g_1$ which are dual with
respect to $B$. Then we claim that
\begin{equation}\label{eq:tech-codiff}
\partial^* \al = \tfrac{1}{2} \textstyle\sum_s\nu_+\wedge 
(i_{[\eta_s,\nu_-]}i_{\xi_s}\al).
\end{equation}
Here we view elements of $\La^*\frak p_+$ as multilinear maps on $\frak g/\frak p$
and for an element $X \in \frak g_{-1}$ we denote by $i_X$ the usual insertion
operator for the element $X+\frak p\in\frak g/\frak p$.

By linearity of $\partial^*$, it suffices to prove \eqref{eq:tech-codiff} for
decomposable elements $\al\in\La^k\frak p_+$, so we can take
$\al=Z_1\wedge\dots\wedge Z_k$. Now $\frak g_2$ is one--dimensional and the bracket
on $\frak p_+=\frak g_1\oplus\frak g_2$ has values in $\frak g_2$ and vanishes if one
of its entries lies in $\frak g_2$. Using the definition of $\partial^*$, we conclude
that $\partial^*\al=0$ if one of the $Z_i$ lies in $\frak g_2$. But since both
$[\eta_s,\nu_-]$ and $\xi_s$ are in $\frak g_{-1}$, the same is true for the right
hand side of \eqref{eq:tech-codiff}. Thus we may restrict to the case that all $Z_i$
are in $\frak g_1$.

Now we can write $[Z_i,Z_j]=B([Z_i,Z_j],\nu_-)\nu_+$ and invariance of $B$ shows that
the numerical factor can be written as $B(Z_i,[Z_j,\nu_-])$. Now using the invariance
of $B$ once more, we can in turn express $[Z_j,\nu_-]\in\frak g_{-1}$ as
$\sum_s-B(Z_j,[\eta_s,\nu_-])\xi_s$. Inserting this, we obtain
$[Z_i,Z_j]=-\sum_s B(Z_i,\xi_s)B(Z_j,[\eta_s,\nu_-])\nu_+$, and observe that $B(Z,X)$
is the value of the linear map defined by $Z$ on $X\in\frak g_{-1}$. Using this, the
claimed formula follows from the definition of $\partial^*$ in \eqref{eq:def-codiff}
by a simple direct computation, thus proving the claim.

To complete the proof, we expand $(\partial^*\al)\wedge\be$ according to
\eqref{eq:tech-codiff} and move the wedge product with $\nu_+$ to obtain summands which are up to a constant multiple of
the form $(-1)^k(i_{[\eta_s,\nu_-]}i_{\xi_s}\al)\wedge\nu_+\wedge\be$. Since
$\left(i_{\xi_s}\al\right)\wedge\nu_+\wedge\be=0$ and $B(\nu_+,[\eta_s,\nu_-])=0$, we see that
this equals $\left(i_{\xi_s}\al\right)\wedge\nu_+\wedge i_{[\eta_s,\nu_-]}\be$. The same argument
shows that the other insertion operator can be moved to $\be$ at the expense of a
sign $(-1)^{k+1}$ and exchanging the two insertion operators causes another
sign--change. Using \eqref{eq:tech-codiff} again, this completes the argument.
\end{proof}

Using this, we are ready to formulate the main compatibility result. 

\begin{prop}\label{prop3.2}
  Let $\Phi:\Om^k(G/P) \to \Om^\ell(G/K)$ be a Poisson transform with corresponding
  Poisson kernel $\phi\in\Om^{(\ell, 2n+1-k)}(G/M)$.
\begin{itemize}
\item[(i)] The compositions $d \circ \Phi$, $\ast \circ \Phi$, $\delta \circ \Phi$
  and $\Delta \circ \Phi$ are again Poisson transforms with associated Poisson
  kernels $d_K\phi$, $\ast_K\phi$, $\delta_K\phi$ and $\Delta_K\phi$, respectively.
\item[(ii)] The compositions $\Cal L \circ \Phi$ and $\Cal L^* \circ \Phi$ are again
  Poisson transforms with associated Poisson kernels $\Cal L_K\phi$ and
  $\Cal L_K^*\phi$, respectively.
\item[(iii)] The compositions $\Phi \circ d$ and $\Phi \circ \partial^*$ are Poisson
  transforms with corresponding Poisson kernels $(-1)^{\ell-k} d_P\phi$ and
  $(-1)^{\ell-k+1} \partial^*_P\phi$, respectively.
\end{itemize}
\end{prop}
\begin{proof}
  (i) Recall from Proposition X in chapter VII of \cite{GHV} that the exterior
  derivative commutes with the fiber integral. Thus, we obtain for all
  $\al \in \Om^k(G/P)$ that
$$
d\Phi(\alpha) = \fint_{G/P} d\left(\phi \wedge \pi_P^*\alpha\right) = \fint_{G/P}
(d_K\phi) \wedge \pi_P^*\alpha,
$$
where we used that $\ph\wedge d\pi_P^*\alpha$ and $(d_P\phi)\wedge \pi_P^*\alpha$
evidently vanish. Next, by tensoriality of the Hodge star it suffices to show the
relation for the composition $\ast \circ \Phi$ at any point, where it can be deduced
from the local description of the fiber integral. Combining those two, we obtain (i)
and together with Proposition IX in chapter VII of \cite{GHV}, (ii) follows readily.

For $\al \in \Om^{k-1}(G/P)$, we get
$d(\phi\wedge\pi_P^*\al)=d\ph\wedge\pi_P^*\al+(-1)^{\ell-k+1}\ph\wedge\pi_P^*d\al$. Applying
$\fint_{G/P}$ the left hand side vanishes, since $d$ commutes with $\fint_{G/P}$. In
the first summand in the right hand side, only $d_P\ph$ leads to a form of the right
bidegree, and we get $0=\fint_{G/P}d_P\ph\wedge\pi_P^*\al+(-1)^{\ell-k+1}\Ph(d\al)$,
which gives the first part of (iii).

For the second part of (iii), we observe that Lemma \ref{lem2.3} and the
definition of $\partial^*_P$ imply that for $\al\in\Om^{k+1}(G/P)$ we get
$$
\phi\wedge\pi_P^*\partial^*\al=\ph\wedge\partial^*_P\pi_P^*\al=
(-1)^{2n+1-k+\ell}(\partial^*_P\ph)\wedge\pi_P^*\al. 
$$ 
Applying $\fint_{G/P}$, the left hand side gives $\Phi(\partial^*\al)$ and the result
follows. 
\end{proof}

Thus we see how to construct Poisson transforms that satisfy the equivalent
conditions of Proposition \ref{prop2.4} and hence descend to the Rumin
complex as discussed in Corollary \ref{cor2.4}: We have to construct
Poisson kernels $\phi$ that satisfy $\partial^*_P\ph=0$ and
$\partial^*_Pd_P\ph=0$. To carry this out explicitly, it will be convenient to pass
to complex valued differential forms. This allows us to decompose forms on $G/K$ into
$(p,q)$--types. Similarly, spitting forms on $G/M$ into bidegrees, the first degree
can be split further into $(p,q)$--types. It follows immediately from the definitions
that $G$--invariance of $\ph$ is equivalent to $G$--invariance of all
$(p,q)$--components of $\ph$ and similarly for vanishing of $\partial^*_P\ph$ and
$\partial^*_Pd_P\ph$.

\subsection{The structure of $M$ and $\frak g/\frak m$.}\label{3.3}
To make things concrete in our case, let us first describe the groups and Lie
algebras we need explicitly.  We realize $G = SU(n+1,1)$ as the group of all complex
matrices $g \in GL(n+2,\Bbb C)$ which satisfy $g^*Sg = S$ and $\det(g) = 1$, where $S$
is the symmetric matrix
\begin{align*}
 S = \begin{pmatrix} 0 & 0 & 1 \\ 0 & \id_n & 0 \\ 1 & 0 & 0 \end{pmatrix}.
\end{align*}
The maximal compact subgroup $K \cong U(n+1)$ of $G$ is given by the fixed points of
the global Cartan involution $g \mapsto \left(g^{-1}\right)^*$. Writing elements in
$G$ as block matrices with the same block sizes as $S$, the minimal parabolic
subgroup $P\subset G$ is given by
\begin{align*}
 P = \left\lbrace \begin{pmatrix} a & -aY^*B & \frac{a}{2}\left(b - |Y|^2\right) \\ 0
   & B & Y \\ 0 & 0 & \overline{a}^{-1}\end{pmatrix} : \begin{array}{c} a \in
   \Bbb C^*, b \in \Bbb C, Y \in \Bbb C^n, B \in U(n), \\ \det(B) =
   \overline{a}a^{-1}\end{array} \right\rbrace.
\end{align*}
Let $G = KAN$ be the Iwasawa decomposition of $G$ with respect to the choices of $K$
and $P$. Then the group $A$ is represented by all block--diagonal matrices in $P$
with $B = \id_n$, $Y=0$, $b=0$ and $a \in \Bbb R^*$,
whereas $N$ corresponds to the elements with $a = 1$ and $B = \id_n$. Finally, the
intersection $M = K \cap P$ consists of all matrices in $P$ with $|a| = 1$, $Y
= 0$ and $b = 0$ and is therefore isomorphic to $S(U(n) \times U(1))$. In this way,
$P = MAN$ is the Langlands decomposition and $G_0 := MA$ is the Levi subgroup of $P$.

Turning to the infinitesimal picture, the Lie algebra $\frak g = \mathfrak{su}(n+1,1)$ of
$G$ is
\begin{align*}
 \frak g = \left\lbrace \begin{pmatrix} b & -Y^* & y \\ X & B & Y \\ x & -X^* &
   -\overline{b}\end{pmatrix} : \begin{array}{c} X, Y \in \Bbb C^n, b \in \Bbb C, x,y
   \in i\Bbb R, B \in \mathfrak{u}(n) \\  b + \tr(B) - \overline{b} = 0  \end{array}
 \right\rbrace. 
\end{align*}
The block form of $\frak g$ defines the $|2|$--grading $\frak g=\frak g_{-2} \oplus
\frak g_{-1} \oplus \frak g_0 \oplus \frak g_1 \oplus \frak g_2$, and the Lie algebra
$\frak p$ of $P$ is $\oplus_{i\geq 0}\frak g_i$. Writing $\xi \in \frak g$ as $\xi =
(x, X, (B, b), Y, y)$ according to this decomposition, the Lie algebra $\frak k$ of
$K$ is given by all elements $(x, X, (B, b), X, x)$ with $b\in i\Bbb R$. In
particular, the Lie algebra $\frak m$ of $M$ consists of all elements of the form
$(0, 0, (B, b), 0, 0)$ with $b \in i\Bbb R$, whereas the Lie algebra $\frak a$ of $A$ is generated by the
grading element $E := (0,0,(0,1),0,0) \in \frak g_0$.

Finally, the Killing form on $\frak g$ is a multiple of the trace form and therefore
determined by its value on the grading element. For nice conventions, it is better to
define $B$ to be $\frac1{2(n+2)}$ times the Killing form, which leads to the
following non-degenerate pairings:
\begin{align*}
 \frak a \times \frak a &\to \Bbb R, & B(E, E) &= 2,\\ \frak m \times \frak m
 &\to \Bbb R, & B((b_1, B_1), (b_2, B_2)) &= (2b_1b_2 + \tr(B_1B_2)), \\ \frak
 g_{-1} \times \frak g_1 &\to \Bbb R, & B(X, Y) &= -2 \langle X, Y
 \rangle,\\ \frak g_{-2} \times \frak g_2 &\to \Bbb R, & B(x,y) &= xy,
\end{align*}
where $\langle \ , \ \rangle$ denotes the standard Hermitian inner product on $\Bbb
C^n$.

Since $M\subset G_0$, we get $\frak m\subset\frak g_0$ and the $|2|$--grading on
$\frak g$ is invariant under the adjoint action of $M$. In particular, as a
representation of $M$, the quotient $\frak g/\frak m$ splits as $\oplus_{i=-2}^2(\frak
g/\frak m)_i$. For $i\neq 0$ we get $(\frak g/\frak m)_i=\frak g_i$, while $(\frak
g/\frak m)_0$ has real dimension one and is spanned by $E+\frak m$. Explicitly,
writing elements of $\frak g/\frak m$ as  $\xi = (x, X, a, Y, y)$ with $x,y\in i\Bbb R$,
$a\in\Bbb R$ and $X,Y\in\Bbb C^n$ according to this identification and viewing $M$ as
$S(U(n) \times U(1))$, the action of $M$ on $\frak g/\frak m$ is given by
$$
(B, b) \cdot \xi = (x, b^{-1}BX, a, b^{-1}BY, y).
$$
We know that $\frak g/\frak m$ is the sum of the horizontal subspace $\frak p/\frak
m$ and vertical subspace $\frak k/\frak m$ and it is easy to identify these subspaces.
The space $\frak p/\frak m$ consists of all elements of the form $(0, 0, a, X, x)$,
while $\frak k/\frak m$ consists of all elements of the form $(x, X, 0, X, x)$. As
stated above, we will work with complex forms to use the decompositions into
$(p,q)$--types, so it will be helpful to deal with the complexification $(\frak
g/\frak m)_{\Bbb C}$. Since $\frak p/\frak m$ is a complex subspace in $\frak g/\frak
  m$, its complexification splits as $(\frak p/\frak m)^{1,0} \oplus (\frak p/\frak
  m)^{0,1}$. Explicitly, the complex structure on $\frak p/\frak m$ maps $(0, 0, a, X, x)$ to $(0,0, \frac{-ix}{2}, JX, 2ia)$. Similarly, the complexification of the CR-subspace $\Bbb H\subset \frak
  k/\frak m$ splits as $\Bbb H^{1,0} \oplus \Bbb H^{0,1}$.

For further computations we fix some notation for elements of
$(\frak g/\frak m)_{\Bbb C}$. First, we put
$$
 Z := \frac{1}{2}(0,0,1,0,2i) \in (\frak p/\frak m)^{1,0} \qquad I:= (i,0,0,0,i) \in (\frak
 k/\frak m)_{\Bbb C}.
$$
On the other hand, for $X\in\Bbb C^n$, we define 
$$
F_X^{1,0} := (0,0, 0, X^{1,0}, 0) \in (\frak p/\frak m)^{1,0} \qquad G_X^{1,0} := (0, X^{1,0}, 0,
X^{1,0}, 0) \in \Bbb H^{1,0},  
$$ and similarly we define $F_X^{0,1}$ and $G_X^{0,1}$, using $X^{0,1}$ instead of
$X^{1,0}$. Of course, we also have $\overline{Z}\in (\frak p/\frak
m)^{0,1}$. Finally, one immediately verifies that the pullback $g_K$ of the
$K$--invariant Hermitian inner product on $\frak g/\frak k$ with the isomorphism to
$\frak p/\frak m$ corresponds to the standard pairing, i.e.~, the non--trivial
pairings are given by
$$
 g_K(F_X^{1,0}, F_Y^{0,1}) = \tfrac12\langle X, Y\rangle \qquad g_K(Z, \overline{Z}) = 1.
$$

 \subsection{Basic invariant forms on $G/M$}\label{3.4}
 It is now easy to construct several invariant one--forms on $G/M$ via $M$--invariant
 elements in $(\frak g/\frak m)_{\Bbb C}^*$ and analyze their exterior
 derivatives. First, we define $I^*\in (\frak g/\frak m)_{\Bbb C}^*$ by requiring
 that $I^*$ vanishes on $(\frak p/\frak m)_{\Bbb C}$ and on $\Bbb H_{\Bbb C}$ and
 that $I^*(I)=1$.

 Recall from Section \ref{3.1} that we split forms on $G/M$ according to
 bidegree. Now we further decompose them according to $(p,q)$--types.  This is no
 problem with respect to the first degree, since this corresponds to multilinear
 functionals on $\frak p/\frak m$, which is a complex subspace of $\frak g/\frak
 m$. Accordingly, it makes sense to say that a form of bidegree $(k,\ell)$ has a
 certain $(p,q)$--type with $p+q=k$, which we will phrase as being of $K$--type
 $(p,q)$. For the second degree, we say that a multilinear form $\om$ of bidegree
 $(k,\ell)$ has $P$--type $(r,s)$ if either $r+s=\ell$, $\om(I)=0$ and
 $\om\in\La^{r,s}\Bbb H^*_{\Bbb C}$ or $r+s=\ell-1$ and $\om=I^*\wedge \tilde\om$
 with $\tilde\om\in\La^{r,s}\Bbb H^*_{\Bbb C}$. We then use the same wording for
 forms on $G/M$.

 In this language, $I^*$ has bidegree $(0,1)$, while its $P$--type is $(0,0)$. On the
 other hand, there are obvious $M$--invariant linear functionals $Z^*$ and
 $\overline{Z}^*$, which both have bidegree $(1,0)$ and $K$--type $(1,0)$ and
 $(0,1)$, respectively. Now we can of course form (partial) exterior derivatives and
 wedge products of these one forms. Recall that for an $M$--invariant $k$--linear,
 alternating functional $\al$ on $(\frak g/\frak m)_{\Bbb C}^k$, the exterior
 derivative of the corresponding invariant $k$--form corresponds to the functional
 that sends $X_0+\frak m,\dots,X_k+\frak m$ to
 $$
 \textstyle\sum_{i<j} (-1)^{i+j} \alpha([X_i, X_j]+ \frak m, X_1+\frak m, \dots, \hat{\imath},
 \dots, \hat{\jmath}, \dots, X_k+\frak m).
 $$
 Using this, we can easily compute the derivatives of $I^*$, $Z^*$ and
 $\overline{Z}^*$ and then correct by wedge products to obtain invariant two--forms
 into which $I$, $Z$, and $\bar Z$ all insert trivially. Explicitly, we define
 $$
\omega_{2,0} := -i\partial_K\overline{Z}^* + i\overline{Z}^* \wedge Z^*, \qquad
\omega_{1,1} := \tfrac{1}{2}d_PZ^* - i Z^* \wedge I^*, \qquad\omega_{0,2} :=
\tfrac{1}{2}d_PI^*,
$$
whose properties are collected in the following table: 
\begin{table}[!htbp]
\begin{center}
\begin{tabular}{|c||c|c|c|c|}
  \hline
  form & bidegree & $K$-type & $P$-type & explicit formula \\
  \hline 
  $\omega_{2,0}$ & $(2,0)$ & $(1,1)$ & $(0,0)$ & $\omega_{2,0}(F_X^{1,0}, F_Y^{0,1}) = -\tfrac{i}2\langle X, Y\rangle$ \\
  $\omega_{1,1}$ & $(1,1)$ & $(1,0)$ & $(0,1)$ & $\omega_{1,1}(F_X^{1,0}, G_Y^{0,1}) = \tfrac12\langle X,Y\rangle$\\
  $\overline{\omega_{1,1}}$ & $(1,1)$ & $(0,1)$ & $(1,0)$ & $\overline{\omega_{1,1}}(F_X^{0,1}, G_Y^{1,0}) = \tfrac12\langle X,Y\rangle$\\
  $\omega_{0,2}$ & $(0,2)$ & $(0,0)$ & $(1,1)$ & $\omega_{0,2}(G_X^{1,0}, G_Y^{0,1}) = -\tfrac{i}2\langle X, Y\rangle$ \\
  \hline
\end{tabular}
\vspace*{2mm}
\caption{Invariant forms of degree $2$.}
\end{center}
\end{table} \newline
Using these invariant $2$-forms
we can write the images of the invariant $1$-forms under the partial derivatives as
\begin{align*}
 \partial_K Z^* &= 0, & d_PZ^* &= 2 \omega_{1,1} + 2i  Z^* \wedge I^*, \\
 \partial_K \overline{Z}^* &= \overline{Z}^* \wedge Z^* + i\omega_{2,0},  & d_P\overline{Z}^* &= 2 \overline{\omega_{1,1}} - 2i \overline{Z}^* \wedge I^*,\\
 \partial_K I^* &= Z^* \wedge I^*,  & d_P I^* &= 2 \omega_{0,2}.
 \end{align*}
In particular, using $d_K = \partial_K + \overline{\partial_K}$ it follows that
\begin{align*}
 d_Kd_P(Z^*) = d_Kd_P(\overline{Z}^*) = 0.
\end{align*} 
  For completeness, we also compute
\begin{align*}
 \partial_K \omega_{2,0} &= - Z^* \wedge \omega_{2,0},  & d_P\omega_{2,0} &= 2i Z^* \wedge \overline{\omega_{1,1}} - 2i \overline{Z}^* \wedge \omega_{1,1},\\
 \partial_K \omega_{1,1} &= 0,  & d_P\omega_{1,1} &= -2i I^* \wedge \omega_{1,1} + 2i Z^* \wedge \omega_{0,2},\\
 \partial_K \overline{\omega_{1,1}} &= -I^* \wedge\omega_{2,0}, & d_P\overline{\omega_{1,1}} &= 2i I^* \wedge \overline{\omega_{1,1}} - 2i \overline{Z}^* \wedge \omega_{0,2}, \\
 \partial_K \omega_{0,2} &= Z^* \wedge \omega_{0,2} - I^* \wedge \omega_{1,1}, &  d_P\omega_{0,2} &= 0.
 \end{align*}
 In particular, there are no new $M$--invariant forms obtained in this way.

 \section{Transforms adapted to the Rumin complex}\label{4}
 Having constructed the basic invariant forms on $G/M$, we can now proceed to
 construct higher degree forms with appropriate properties, and thus Poisson
 transforms for differential forms on a CR--sphere that descend to the Rumin
 complex.

 \subsection{The operator $\partial^*_P$.}\label{4.1}
 As discussed in the end of Section \ref{3.2}, a crucial role for verifying the
 conditions on a Poisson kernel from Corollary \ref{cor2.4} is played by the
 operator $\partial^*_P$ on differential forms on $G/M$ defined in that section. For
 invariant forms this operator corresponds to a map on alternating multilinear forms
 on $(\frak g/\frak m)_{\Bbb C}^*$ that we denote by the same symbol. Thus we start
 by computing the latter operator using the notation from Section \ref{3.4}.

  \begin{prop}\label{prop4.1}
 Let $\{e_1, \dotsc, e_n\}$ be an orthonormal basis of $\Bbb C^n$ with
 respect to the standard Hermitian inner product. Then for
 $\al\in\La^k(\frak g/\frak m)_{\Bbb C}^*$, and up to a non--zero
 multiple, we have
 \begin{align*}
   \partial^*_P\al=\textstyle\sum_{s=1}^n I^* \wedge \iota_{G_{e_s}^{0,1}}
   \iota_{G_{e_s}^{1,0}}\al,
 \end{align*}
 with $\iota$ denoting insertion operators. In particular, the operator $\partial^*_P$
 vanishes on the ideal generated by $I^*$.
\end{prop}
\begin{proof}
The definition of $\partial^*_P$ in Section \ref{3.2} was via the Kostant
codifferential $\partial^*$ and the identification of $\frak k/\frak m$ with $\frak
g/\frak p$. Thus we start by rewriting the formula \eqref{eq:tech-codiff} from the
proof of Lemma \ref{lem2.3} in terms of the basis of $\frak p$ induced by $\{e_1,
\dots, e_n\}$.

For all $X \in \Bbb C^n$, we denote by $\xi_X$ and $\eta_X$ the corresponding
elements in $\frak g_1$ and $\frak g_{-1}$, respectively. By $\nu_{\pm}$ we denote
the elements in $\frak g_{\pm 2}$ corresponding to $i$. Then $\{\xi_{e_s},
\xi_{ie_s}, \nu_+\}$ is a real basis of $\frak p_+$, and from Section \ref{3.3} we
deduce that $\{-\tfrac12\eta_{e_s}, -\tfrac12\eta_{ie_s}, -\nu_-\}$ is the dual basis
of $\frak g_-$ with respect to $B$. Moreover, the Lie brackets $[\xi_{e_s}, \nu_-]$
and $[\xi_{ie_s},\nu_-]$ equal $-\eta_{ie_s}$ and $\eta_{e_s}$, respectively. Thus,
for all $\be\in \Lambda^k \frak p_+$ formula \eqref{eq:tech-codiff} says that
$\partial^*\be$ is a non--zero multiple of 
$$ \textstyle\sum_{s = 1}^{n} \nu_+ \wedge \iota_{\eta_{ie_{s}}}
\iota_{\eta_{e_s}}\beta.
$$

To obtain the corresponding expression for $\partial^*_P$, we just have to interpret
this in terms of the $M$-module $\frak g/\frak m$, so we identify $\frak g_-$ with
$\frak k/\frak m$ and $\frak p_+$ with $(\frak k/\frak m)^*$ via the Killing
form. This readily shows that $\eta_{e_s}$ corresponds to the element $G_{e_s}$ while
$\nu_+$ corresponds to $I^*$ up to a non--zero factor. Therefore, we obtain for all
$\al \in \Lambda^{0,k}(\frak g/\frak m)^*$ that $\partial^*_P\al$ is a nonzero
multiple of 
\begin{align*}
 \textstyle\sum_{s=1}^n I^* \wedge \iota_{G_{ie_{s}}}\iota_{G_{e_s}}\al,
\end{align*}
and using the definition of $\partial^*_P$ on decomposable elements as well as
linearity this continues to hold for all elements in $\Lambda^{\ell, k}(\frak g/\frak
m)^*$. The claimed formula then holds by decomposing the basis vectors $G_{e_s}$ into
holomorphic and antiholomorphic parts.
\end{proof}

\subsection{Invariant forms of higher degree}\label{4.2}
Now we can start building up forms of higher degree from the basic invariant
two--forms introduced in Section \ref{3.4}. We do this in a notation that expresses
the bidegree as well as the $K$--type and the $P$--type.

\begin{definition}\label{def4.2}
  Let $p, q, k$ be non-negative integers such that $0 \le p, q, k-p, k-q \le n$. For
  all $\max\{0, p+q-k\} \le j \le \min\{p,q\}$ we define
$$
\omega_j^{p,q;k} := \omega_{2,0}^j \wedge \omega_{0,2}^{k-(p+q)+j} \wedge
\omega_{1,1}^{p-j} \wedge \overline{\omega_{1,1}}^{q-j}.
$$
\end{definition}

By construction, each of these forms vanishes upon insertion of $I$, $Z$ and
$\overline{Z}$. Moreover, from Table 2 we readily see that $\omega_j^{p,q;k}$ has
bidegree $(p+q,2k-p-q)$, $K$-type $(p,q)$, and $P$-type $(k-p,k-q)$. Since these
types conversely determine $p$, $q$ and $k$, forms with different values of these
parameters are automatically linearly independent (if non-zero). For later use, we
next show that for fixed $p$, $q$ and $k$ with $k>n$, there is a linear relation
between the forms $\omega_j^{p,q;k}$ for different values of $j$. 

\begin{prop}\label{prop_relation_omega}
  For all $0 \le p, q \le n$ we define
$$
\kappa_j^{p,q;k} := \textstyle\binom{k}{p} \binom{p}{j} \binom{k-p}{q-j},
$$
where we agree that $\kappa_j^{p,q;k} = 0$ if one of the binomial coefficients is not
defined. Then for $k>n$, we get
$$
\textstyle\sum_{j} \kappa_j^{p,q;k} \omega_j^{p,q;k} = 0.
$$
\end{prop}
\begin{proof}
  Since $k>n$ the form $\omega_{0,2}^k$ is trivial. Hence for all $X_1, \dots, X_p$
  and $Y_1, \dots, Y_q \in \Bbb C^n$, we obtain
\begin{equation}\label{id0}
\iota_{G_{Y_q}^{0,1}} \dots \iota_{G_{Y_1}^{0,1}} \iota_{G_{X_p}^{1,0}}\dots
\iota_{G_{X_1}^{1,0}}\omega_{0,2}^k = 0.
\end{equation}
Now our definitions easily imply that for all $X \in \Bbb C^n$ we get
$$
\iota_{G_X^{1,0}}\omega_{0,2} = i\iota_{F_X^{1,0}}\omega_{1,1} \qquad
\iota_{G_X^{0,1}}\omega_{0,2} = -i\iota_{F_X^{0,1}}\overline{\omega_{1,1}} \qquad
\iota_{G_X^{0,1}}\omega_{1,1} = -i\iota_{F_X^{0,1}}\omega_{2,0}
$$
Inductively, the first of these relations readily implies that, up to a non--zero
constant, we get 
$$
\iota_{G_{X_p}^{1,0}}\dots\iota_{G_{X_1}^{1,0}}\omega_{0,2}^k=\iota_{F_{X_p}^{1,0}}\dots
\iota_{F_{X_1}^{1,0}}\omega_{1,1}^p\omega_{0,2}^{k-p}.
$$
The remaining two relations then show that, up to a factor $(-i)$, we get
$$
\iota_{G_{Y_1}^{0,1}}\omega_{1,1}^p\omega_{0,2}^{k-p}=\iota_{F_{Y_1}^{0,1}}(p\om_{2,0}\om_{1,1}^{p-1}\om_{0,2}^{k-p}+
(k-p)\overline{\om_{1,1}}\om_{1,1}^p\om_{0,2}^{k-p-1}). 
$$
Since insertion operators always anti-commute, we conclude inductively, that
\eqref{id0} can be rewritten as 
$$
0=\iota_{F_{Y_q}^{0,1}} \dots \iota_{F_{Y_1}^{0,1}} \iota_{F_{X_p}^{1,0}}\dots
\iota_{F_{X_1}^{1,0}}(\textstyle\sum_j\kappa_j^{p,q;k} \omega_j^{p,q;k}).
$$
But by construction $\sum_j\kappa_j^{p,q;k} \omega_j^{p,q;k}$ is a form that vanishes
under insertion of $I$, $Z$, and $\overline{Z}$ and has $K$--type $(p,q)$, so the
equation shows that it vanishes upon insertion of any $p+q$ tangent vectors from
$T'$. But again by construction it has bidegree $(p+q,2k-p-q)$, so the claim follows.
\end{proof}

The coefficients showing up in these relations are characterized by a linear
recursion that we discuss next. 

\begin{lemma}\label{lem_recursion_kappa}
  The coefficients $\ka_j:=\kappa_j^{p,q;k}$ defined in Proposition
  \ref{prop_relation_omega} satisfy the following system of linear
  relations, which determines them up to a non-zero constant. 
\begin{equation}\label{eqn_relations_kappa}
  \kappa_j(p-j)(q-j) = \kappa_{j+1}(j+1)(k-(p+q)+j+1) \quad \text{for all\ } j \in \Bbb
  Z.
\end{equation}
\end{lemma}
\begin{proof}
  One immediately verifies that the numbers $\kappa_j^{p,q;k}$ satisfy the relations
  \eqref{eqn_relations_kappa}. Conversely, let $(\kappa_j)_{j \in \Bbb Z}$ be
  another solution and put $j_0 := \min\{0, p+q-k\}$.  Then for $j = j_0-1$, the
  right hand side of \eqref{eqn_relations_kappa} is zero, so
  $\kappa_{j_0-1} = 0$ and inductively $\kappa_{j} = 0$ for all $j < j_0$. Similarly,
  for $j = \min\{p,q\}$ the left hand side of \eqref{eqn_relations_kappa}
  vanishes, whence $\kappa_j = 0$ for all $j > \min\{p,q\}$. Therefore, the sequence
  $(\kappa_j)_{j \in \Bbb Z}$ is possibly nontrivial only for
  $j_0 \le j \le \min\{p,q\}$, and since the relations are linear it follows that
  $\kappa_j$ is a constant multiple of $\kappa_j^{p,q;k}$.
\end{proof}

Finally, we can use Proposition \ref{prop4.1} to compute the action of $\partial^*_P$
on the forms $\om^{p,q;k}_j$:

\begin{lemma}\label{lem4.1} 
  Up to a nonzero constant that is independent of $p$, $q$, $k$ and $j$, we have
\begin{align*}
  \partial^*_P\omega_j^{p,q;k} = & (k-(p+q)+j)(k-j-(n+1)) I^* \wedge \omega_j^{p,q;k-1} \\ 
  &+ (p-j)(q-j) I^* \wedge \omega_{j+1}^{p,q;k-1}.
\end{align*}
\end{lemma}
\begin{proof}
  In view of Proposition \ref{prop4.1} we mainly have to compute
  $\sum_s\iota_{G_{e_s}^{0,1}} \iota_{G_{e_s}^{1,0}}\om^{p,q;k}_j$. Using the
  derivation property of insertion operators and the definitions, we readily see that
$$
\iota_{G_{e_{s}^{1,0}}}\om^{p,q;k}_j=(k-p-q+j)(\iota_{G_{e_s}^{1,0}}\om_{0,2})\om^{p,q;k-1}_j+
(q-j)(\iota_{G_{e_s}^{1,0}}\overline{\om_{1,1}})\om^{p,q-1;k-1}. 
$$
The image of the right hand side under $\iota_{G_{e_s}^{0,1}}$ can again be computed
using the derivation property of insertion operators and the definitions. This time,
the expressions involve the one-forms like $\iota_{G_{e_s}^{0,1}}\om_{0,2}$ and
$\iota_{G_{e_s}^{0,1}}\om_{1,1}$. After summation over $s$ the wedge products of two
one-forms occurring in the summands can be re-expressed in terms of the basic
two-forms and the result follows from a direct computation.
\end{proof}

\subsection{The case $p+q\leq n$}\label{4.3}
We will now start the construction of Poisson kernels such that the associated
Poisson transforms descend to the Rumin complex as described in Corollary
\ref{cor2.4}. We focus on the case of transforms that preserve the degree of
differential forms and then for a transform defined on $\Ga(\Cal H_\ell)$ we can
restrict to the case of a transform with values in $\Om^{p,q}(G/K)$ for some fixed
type $(p,q)$ such that $p+q=\ell$. To obtain such a transform, we have to construct a
Poisson kernel $\ph_{p,q}$ of bidegree $(p+q, 2n+1-(p+q))$ and $K$-type $(p,q)$. To
ensure that the transform descends to the Rumin complex, we need that
$\partial^*_P\ph_{p,q}=0$ and $\partial^*_Pd_P\ph_{p,q}=0$.

In Proposition \ref{prop_homology_bundles}, we have seen that the nature of the
homology bundles $\Cal H_\ell$ is different for $\ell\leq n$ and $\ell>n$. Thus it is
not surprising that the appropriate choices for the kernels $\ph_{p,q}$ look rather
different in the cases $p+q\leq n$ and $p+q>n$, and we will discuss these cases
separately starting with the former. Denoting by $H\subset T(G/P)$ the
contact subbundle and putting $Q:=T(G/P)/H$, we know that in this case the homology
bundle $\Cal H_\ell$ is a subbundle of $\La^\ell H^*$. To get a transform that is
non-zero on this homology bundle, we have to ensure that the kernel $\ph_{p,q}$
satisfies $\iota_I \ph_{p,q} \neq 0$ and $I^* \wedge \ph_{p,q} = 0$. Hence it is a
natural idea to try constructing $\ph_{p,q}$ as a wedge product of $I^*$ with some
invariant form $\pi$ of bidegree $(p+q, 2n-(p+q))$ and $K$-type $(p,q)$. By
Proposition \ref{prop4.1}, this automatically implies that $\partial^*_P\ph_{p,q}=0$.

Since we also have to ensure that $\partial^*_Pd_P\ph_{p,q}=0$, it is natural to try
using a $d_P$-exact form $\pi$ in the above construction. For $0\leq j\leq
\min\{p,q\}$ define
\begin{equation}\label{pi-def1} 
\pi_j^{p,q;k} := (d_PI^*)^{k-(p+q)} \wedge (2d_Kd_P\omega_{1,1})^j \wedge
(d_PZ^*)^{p-j} \wedge (d_P\overline{Z}^*)^{q-j},
\end{equation} 
which is of bidegree $(p+q,2k-(p+q))$, of $K$-type $(p,q)$ and $d_P$--exact. In
particular, any linear combination of the forms $\pi_j^{p,q;n}$ can be wedged with
$I^*$ to obtain a form $\ph_{p,q}$ as above. The definition readily implies that
$d_P(I^*\wedge\pi_j^{p,q;k})=\pi_j^{p,q;k+1}$ (which explains why we do not restrict
our considerations to the case $k=n$).

\begin{thm}\label{thm_main1}
  For $p+q\leq n$ define
  $\ph_{p,q}:= \sum_{j=0}^{\min\{p,q\}}\ka_j^{p,q;n+1}I^* \wedge\pi_j^{p,q;n}$, where
  the constants $\ka_j^{p,q;n+1}$ are defined in Proposition
  \ref{prop_relation_omega}. Then this kernel gives rise to a Poisson transform
  $\Ph:\Om^{p+q}(G/P,\Bbb C)\to\Om^{p,q}(G/K)$ which satisfies the conditions of
  Proposition \ref{prop2.4} and thus factorizes to the Rumin complex as described in
  Corollary \ref{cor2.4} and has harmonic values.
\end{thm}
\begin{proof}
  It remains to prove that $\partial^*_Pd_P\ph_{p,q}=0$. To do this, we first
  observe that by definition $d_P(I^*\wedge\pi_j^{p,q;n})=\pi_j^{p,q;n+1}$. To
  compute $\partial^*_P\pi_j^{p,q;n+1}$ we observe that for any form $\tau$, we get
  $\iota_I(I^*\wedge\tau)=\tau-I^*\wedge \iota_I\tau$. Thus Proposition \ref{prop4.1}
  shows that $\partial_P^*\tau=\partial_P^*\iota_I(I^*\wedge\tau)$. But the formulae
  in the end of Section \ref{3.4} imply that, up to elements in the ideal generated
  by $I^*$, we can replace $d_PZ^*$ by $2\om_{1,1}$ and $d_P\overline{Z}^*$ by
  $2\overline{\om_{1,1}}$ and $2d_Kd_P\omega_{1,1}$ by
  $4\om_{2,0}\wedge\omega_{0,2}$. This then shows that
  $I^*\wedge\pi_j^{p,q;k}=2^k I^* \wedge \om_j^{p,q;k}$ and thus
  $\partial_P^*\pi_j^{p,q;k}=2^k\partial_P^*\om_j^{p,q;k}$. But this implies
  $\partial^*_Pd_P\ph_{p,q}=2^{n+1}\sum_j\ka_j^{p,q;n+1}\partial^*_P\om_j^{p,q;n+1}$,
  which vanishes since $\sum_j\ka_j^{p,q;n+1}\om_j^{p,q;n+1}=0$ by Proposition
  \ref{prop_relation_omega}.
\end{proof}

\subsection{The case $p+q>n$}\label{4.4}
The case of high degree forms is significantly more complicated. For $\ell\geq n+1$,
the homology bundle $\Cal H_\ell$ is a quotient of $\La^{\ell-1}H^*\otimes Q^*$ by
Proposition \ref{prop_homology_bundles}. Consequently, for an appropriate Poisson
kernel $\ph_{p,q}$ we certainly must have $I^*\wedge \ph_{p,q}\neq 0$, so compared to
the low-degree case, already satisfying $\partial_P^*\ph_{p,q} = 0$ becomes a non-trivial
problem. The basic strategy to construct appropriate forms will again be to define a
family of $d_P$-closed forms with non-trivial wedge product with $I^*$ and construct
$\ph_{p,q}$ as a linear combination of wedge products of these with appropriate
invariant one-forms. Explicitly, for $p+q-k\leq j\leq\min\{p,q\}$, we define
\begin{equation}\label{tpi-def}
\tilde\pi_j^{p,q;k}:=(2id_KZ^*)^j\wedge (d_PI^*)^{k-(p+q)+j}\wedge
(d_PZ^*)^{p-j}\wedge (d_P\overline{Z}^*)^{q-j}. 
\end{equation} 
From Section \ref{3.4}, we know that $d_Pd_KZ^*=0$, so all these forms are
$d_P$-closed and visibly $\tilde\pi_j^{p,q;k}$ is of bidegree $(p+q,2k-(p+q))$ and of
$K$-type $(p,q)$. Using these, we can now formulate our second main theorem. 

\begin{thm}\label{thm_main2}
  For $p+q\geq n+1$ and $\alpha$, $\beta \in \mathbb{C}$ define
\begin{align*}
\phi_{p,q}^{\alpha, \beta} := \phi_{p,q} := \textstyle\sum_{j\in\Bbb Z} &\Bigl( \al_j Z^*\wedge\tilde\pi_j^{p-1,q;n} + \beta_j \overline{Z}^* \wedge \tilde\pi_j^{p,q-1;n} \\ &+ 
\gamma_j I^* \wedge \tilde\pi_j^{p,q;n} + 2i \delta_j  I^* \wedge Z^* \wedge \overline{Z}^*
\wedge\tilde\pi_j^{p-1,q-1;n-1}\Bigr),
\end{align*}
where $\alpha_j:=\alpha\kappa_{j+1}^{p,q+1;n+1}$, $\beta_j := \beta \kappa_{j+1}^{p+1,q;n+1}$, 
$\gamma_j:=\tfrac{\alpha(p+1) + \beta(q+1)}{p+q-n} \kappa_{j+1}^{p+1,q+1;n+1}$ and
$\delta_j:=-\tfrac{(n+1)(\alpha(n+1-q)+ \beta(n+1-p))}{p+q-n} \kappa_{j+1}^{p,q;n}$, with the constants $\ka$
defined in Proposition \ref{prop_relation_omega}. Then these kernels give rise to a $2$-parameter family of
Poisson transforms $\Ph^{\alpha, \beta}:\Om^{p+q}(G/P)\to\Om^{p,q}(G/K)$ which satisfy the
conditions of Proposition \ref{prop2.4} and thus factorize to the Rumin complex as
described in Corollary \ref{cor2.4} and have harmonic values.
\end{thm}
\begin{proof}
  We have to prove that $\partial^*_P\ph_{p,q}=0$ and
  $\partial^*_Pd_P\ph_{p,q}=0$. As noted in the proof of Theorem \ref{thm_main1}, we
  get $\partial^*_P\tau=\partial^*_Pi_I(I^*\wedge\tau)$ which will be used
  heavily. Using the fact that, up to elements in the ideal generated by $Z^*$, the
  form $2id_KZ^*$ is congruent to $2\om_{2,0}$, similar arguments as in the proof of
  Theorem \ref{thm_main1} show that up to a multiple
  $I^*\wedge Z^*\wedge\tilde\pi_j^{p-1,q;k}$ equals
  $I^*\wedge Z^*\wedge\om^{p-1,q;k}_j$ as well as that
  $I^* \wedge \overline{Z}^* \wedge \tilde\pi_j^{p,q-1;k}$ coincides with
  $I^* \wedge \overline{Z}^* \wedge\om^{p,q-1;k}_j$.  Next, from Proposition
  \ref{prop4.1} we conclude that
  $\partial^*_P(Z^*\wedge\om^{p-1,q;k}_j)=-Z^*\wedge\partial^*_P \om^{p-1,q;k}_j$ and
  similarly
  $\partial^*_P(\overline{Z}^* \wedge \om^{p,q-1;k}_j) = - \overline{Z}^*
  \wedge \partial_P^* \om^{p,q-1;k}_j$.
  Since we can ignore forms that contain $I^*$ in computing the image of $\ph_{p,q}$
  under $\partial^*_P$ we conclude that $\partial_P^*\phi_{p,q}$ coincides with
$$
\textstyle\sum_{j\in\Bbb
  Z}\left(\al_jZ^*\wedge\partial^*_P\om^{p-1,q;n}_j + \beta_j \overline{Z}^* \wedge \partial_P^*\om^{p,q-1;n}_j \right)
$$
up to a multiple. Now using Lemma \ref{lem4.1}, we see that the right hand side can be written as
$\sum_ja_jZ^*\wedge I^*\wedge\om^{p-1,q;n-1}_j + b_j \overline{Z}^* \wedge I^* \wedge \omega_j^{p,q-1;n-1}$, and inserting for $\al_j$ and $\beta_j$ we see that,
up to an overall constant,
\begin{align*}
a_j &=-\kappa_{j+1}^{p,q+1;n+1}(n-p-q+j+1)(j+1)+\ka^{p,q+1;n+1}_j(p-j)(q-j+1), \\
b_j &= -\kappa_{j+1}^{p+1,q;n+1}(n-p-q+j+1)(j+1)+\ka^{p+1,q;n+1}_j(p-j+1)(q-j),
\end{align*} 
which both vanish bye Lemma \ref{lem_recursion_kappa}.

To compute $d_P\ph_{p,q}$, we first observe that definition \eqref{tpi-def}
readily implies that $d_P(Z^*\wedge\tilde\pi^{p-1,q;n}_j)$, $d_P(\overline{Z}^* \wedge \tilde\pi_j^{p,q-1;n})$ and
$d_P(I^*\wedge\tilde\pi^{p,q;n}_j)$ are all equal to $\tilde\pi^{p,q;n+1}_j$. On the
other hand, modulo the ideal generated by $I^*$,
$d_P(I^* \wedge Z^* \wedge \overline{Z}^* \wedge\tilde\pi_j^{p-1,q-1;n-1})$ is
congruent to $Z^* \wedge \overline{Z}^* \wedge\tilde\pi_j^{p-1,q-1;n}$, implying 
\begin{equation}\label{main}
\partial^*_Pd_p\ph_{p,q}=\textstyle\sum_j\Bigl((\al_j+\be_j + \gamma_j)\partial^*_P\tilde\pi^{p,q;n+1}_j+ 
2i\delta_j\partial^*_P(Z^* \wedge \overline{Z}^* \wedge\tilde\pi_j^{p-1,q-1;n})\Bigr). 
\end{equation}
From Section \ref{3.4} we know that $d_KZ^*=Z^*\wedge\overline{Z}^*-i\om_{2,0}$ and
using this we conclude as in the proof of Theorem \ref{thm_main1} that 
\begin{equation}\label{tech1}
  I^*\wedge \tilde\pi^{p,q;n+1}_j=2^{n+1}I^*\wedge(\om^{p,q;n+1}_j+ijZ^* \wedge
  \overline{Z}^* \wedge\om^{p-1,q-1;n}_{j-1}).
\end{equation}  
Thus in the right hand side of \eqref{main}, we may replace $\tilde\pi^{p,q;n+1}_j$
by $2^{n+1}(\om^{p,q;n+1}_j+ijZ^* \wedge \overline{Z}^* \wedge\om^{p-1,q-1;n}_{j-1})$. But
\eqref{tech1} also implies that
$$
  I^*\wedge Z^* \wedge \overline{Z}^* \wedge\tilde\pi_j^{p-1,q-1;n}=2^{n}I^*\wedge
  Z^* \wedge \overline{Z}^* \wedge \om_j^{p-1,q-1;n}.
$$
From Proposition \ref{prop4.1} we next conclude that
$\partial^*_P(Z^* \wedge \overline{Z}^*\wedge\tau)=Z^* \wedge
\overline{Z}^*\wedge\partial^*_P\tau$
for any form $\tau$. Hence in the second term in the right hand side of \eqref{main},
we can replace $\tilde\pi_j^{p-1,q-1;n}$ by $2^{n}\om_j^{p-1,q-1;n}$. Reordering
the sum, we obtain
\begin{align}\label{main2}
\partial^*_Pd_P\ph_{p,q}=2^{n+1}\textstyle\sum_j &\Bigl(\tilde{\gamma}_j\partial^*_P\om^{p,q;n+1}_j +\tilde{\delta}_j iZ^* \wedge \overline{Z}^* \wedge \partial^*_P\om^{p-1,q-1;n}_j\Bigr), 
\end{align}
where $\tilde{\gamma}_j := \alpha_j + \beta_j + \gamma_j$ and $\tilde{\delta}_j := (j+1)\tilde{\gamma}_{j+1}+\delta_j$. Now a direct computation shows that the coefficients in the above sum are given by
\begin{align*}
 \tilde{\gamma}_j &= \frac{\alpha(n+1-q) + \beta(n+1-p)}{p+q-n} \kappa_{j}^{p,q;n+1}, & 
 \tilde{\delta}_j &= - (n+1-(p+q))\tilde{\gamma}_{j+1}.
\end{align*}
Applying the formula for the $P$-codifferential from Lemma \ref{lem4.1} to the first sum we deduce that the coefficient of $I^* \wedge \omega_{j+1}^{p,q;n}$ is given by
\begin{align*}
 -(j+1)(n-(p+q)+j+2)\tilde{\gamma}_{j+1} + (p-j)(q-j)\tilde{\gamma}_j,
\end{align*}
which is trivial due to Lemma \ref{lem_recursion_kappa}. Similarly, using again the formula from Lemma \ref{lem4.1} the coefficient of $I^* \wedge Z^* \wedge \overline{Z}^* \wedge \omega_{j+1}^{p-1,q-1;n-1}$ is given by
\begin{align*}
 -(j+2)(n-(p+q)+j+3)\tilde{\delta}_{j+1} + (p-j-1)(q-j-1)\tilde{\delta}_j,
\end{align*}
so since $\tilde{\delta}_j$ is a constant multiple of $\kappa_{j+1}^{p,q;n+1}$ this is again trivial due to Lemma \ref{lem_recursion_kappa}. All in all we see that $\partial^*_Pd_P\varphi_{p,q} = 0$.
\end{proof}

\subsection{Properties of the image of the Poisson transforms}\label{4.5}
In the next step we analyze the properties of the images of the Poisson transforms
$\Phi$ constructed in Theorems \ref{thm_main1} and \ref{thm_main2} with respect
to several differential operators on $G/K$. By Proposition \ref{prop3.2} it suffices
to apply the corresponding $M$-equivariant maps to the underlying Poisson kernels
$\phi_{p,q}$. Since for their construction we had to distinguish the cases
$p+q \le n$ and $p+q > n$, their properties will also differ significantly in these
cases.

\begin{prop}\label{prop_properties_below_half} 
 For $p+q \le n$ let $\Phi \colon \Omega^{p+q}(G/P, \mathbb{C}) \to \Omega^{p,q}(G/K)$ be the Poisson transform constructed in Theorem \ref{thm_main1}. Then the image of $\Phi$ is contained in the space of harmonic, coclosed and primitive differential forms on $G/K$. 
 \end{prop}
 \begin{proof} 
   Recall from Theorem \ref{thm_main1} that the Poisson kernel underlying $\Phi$ is
   given by
   $\phi_{p,q} = \sum_{j = 0}^{\min\{p,q\}} \kappa_j^{p,q;n+1} I^*\wedge
   \pi_j^{p,q;n}$,
   where the forms $\pi_j^{p,q;n}$ are defined by \eqref{pi-def1} and the constants
   $\kappa_j^{p,q;n+1}$ are defined in Proposition \ref{prop_relation_omega}. By
   Proposition \ref{prop3.2} it suffices to show that this kernel is contained in the
   kernel of $\delta_K$ and $\mathcal{L}_K^*$.
  
   We have seen in the proof of Theorem \ref{thm_main1} that the forms
   $\pi_j^{p,q;k}$ satisfy the relation
   $I^* \wedge \pi_j^{p,q;k} = 2^k I^* \wedge \omega_j^{p,q;k}$. Hence, by linearity
   we can immediately apply the formulae for the adjoint of the $K$-Lefschetz map
   from Proposition \ref{prop_adjoint_K_Lefschetz}. After shifting the index
   appropriately we obtain that
   $\Cal L_K^*\phi_{p,q} = 4\sum_{j} \tilde{\kappa}_j I^* \wedge \pi_j^{p-1,q-1;n-1}$
   with coefficients
  \begin{align*}
  \tilde{\kappa}_j = (j+1)(n+1-(p+q)+j+1) \kappa_{j+1}^{p,q;n+1} - (p-j)(q-j)\kappa_j^{p,q;n+1}.
\end{align*}
   However, due to equation \eqref{eqn_relations_kappa} these constants are trivial.
   
   Similarly, applying the formulae for $\partial_K^*$ from Proposition
   \ref{prop_formula_K_codifferential} to the Poisson kernel yields
   $\partial_K^*\phi_{p,q} = 4i\sum_{j} \tilde{\kappa}_j I^* \wedge \overline{Z}^*
   \wedge \pi_j^{p-1,q-1;n-1} = 0.$
   Furthermore, the complex conjugate of $\phi_{p,q}$ equals $\phi_{q,p}$, which
   readily implies
   $\overline{\partial}_K^*\phi_{p,q}= \overline{\partial_K^*\phi_{q,p}} = 0$ and
   hence also $\delta_K\phi_{p,q} = 0$.
 \end{proof}

 Next, we determine properties of the $2$-parameter family of Poisson transforms for the case $p+q > n$ constructed in Theorem \ref{thm_main2}. 
 
 \begin{prop}\label{prop_properties_above_half} For $p+q > n$ and $\alpha$,
   $\beta \in \mathbb{C}$ the image of the Poisson transform
   $\Ph_{p,q}^{\alpha, \beta}$ constructed in Theorem \ref{thm_main2} is contained in
   the space of coprimitive forms. Moreover,
   $\partial^* \circ \Phi^{\alpha, \beta}_{p,q}$ is independent of $\beta$,
   $\overline{\partial}^* \circ \Phi^{\alpha, \beta}_{p,q}$ is independent of
   $\alpha$ and these satisfy
 \begin{align*}
   \partial^* \circ \Phi^{\alpha, *}_{p+1,q} &= \overline{\partial}^* \circ \Phi^{*, \alpha}_{p,q+1}.
 \end{align*}
 Furthermore, the partial derivatives are related via
 \begin{align*}
  -2i(n-p)\partial \circ \Phi_{p,q}^{\alpha, \beta} &= (p+q-n+1)\Phi^{0, \beta}_{p+1,q} \circ d, \\
  2i(n-q)\overline{\partial} \circ \Phi_{p,q}^{\alpha, \beta} &= (p+q-n+1)\Phi^{\alpha, 0}_{p,q+1} \circ d.
 \end{align*}
\end{prop}

\begin{proof} For the first part we need to show that the wedge product of
  $\phi_{p,q}$ with the pullback $\omega_M$ of the K\"{a}hler form on $G/K$ along the
  canonical projection is trivial. From \eqref{eq:pullback_Kaehler_form} we know that
  $\omega_M = \frac{1}{2}(\omega_{2,0} + iZ^* \wedge \overline{Z}^*)$, which
  coincides with $\tfrac{i}{2}d_KZ^*$. Consequently, by definition of
  $\tilde{\pi}_j^{p,q;k}$ in \eqref{tpi-def} it readily follows that
  $\omega_M\wedge \tilde{\pi}_j^{p,q;k} = \frac{1}{4}
  \tilde{\pi}_{j+1}^{p+1,q+1;k+1}$.
  Therefore, we can write $\omega_M \wedge \phi_{p,q}$ up to an overall constant as
  \begin{align}\label{eq:wedge_product_Kaehler_Poisson_1}
  \begin{aligned}
\textstyle\sum_{j\in\Bbb Z} &\Bigl( \al_j Z^*\wedge\tilde\pi_{j+1}^{p,q+1;n+1} + \beta_j \overline{Z}^* \wedge \tilde\pi_{j+1}^{p+1,q;n+1} \\ &+ 
\gamma_j I^* \wedge \tilde\pi_{j+1}^{p+1,q+1;n+1} + 2i \delta_j  I^* \wedge Z^* \wedge \overline{Z}^*
\wedge\tilde\pi_{j+1}^{p,q;n}\Bigr)
\end{aligned}
\end{align}
with the coefficients defined in Theorem \ref{thm_main2}.  Next, using the formulae
for the basic invariant forms a direct computation shows that
$Z^* \wedge \tilde{\pi}_{j+1}^{p,q+1;n+1}$ coincides with
$2^{n+1}(Z^* \wedge \omega_{j+1}^{p,q+1;n+1} - i(q-j) I^* \wedge Z^* \wedge
\overline{Z}^* \wedge \omega_{j+1}^{p,q;n})$.
Now if we multiply the first summand by $\alpha_j = \alpha \kappa_{j+1}^{p,q+1;n+1}$
and sum over all $j \in \mathbb{Z}$ the result vanishes due to Proposition
\ref{prop_relation_omega}. Therefore, up to an overall constant we can replace in
\eqref{eq:wedge_product_Kaehler_Poisson_1} the form
$Z^* \wedge \tilde{\pi}_{j+1}^{p,q+1;n+1}$ with
$-i(q-j) I^* \wedge Z^* \wedge \overline{Z}^* \wedge
\omega_{j+1}^{p,q;n}$.
Proceeding similarly for the other summands in
\eqref{eq:wedge_product_Kaehler_Poisson_1} we can write $\omega_M \wedge \phi_{p,q}$
up to an overall constant as
$\sum_j a_j I^* \wedge Z^* \wedge \overline{Z}^* \wedge \omega_{j+1}^{p,q;n}$ with
coefficients
\begin{align*}
  a_j := -\alpha_j(q-j) - \beta_j(p-j) + (j+2)\gamma_{j+1} + \delta_j. 
\end{align*}
But a direct computation yields $a_j = 0$ for all $j$ and therefore
$\omega_M \wedge \phi_{p,q} = 0$.
  
Next, we determine the image of $\phi_{p,q}$ under $\partial_K$. Using the formulae for the partial derivatives of the basic invariant forms from Section \ref{3.4} a direct computation yields 
\begin{align}\label{eq:K_derivative_t_pi}
 \partial_K \tilde{\pi}_j^{p,q;k} = (k-(p+q)+j)(Z^* \wedge \tilde{\pi}_j^{p,q;k} - I^* \wedge \tilde{\pi}_j^{p+1,q;k}).
\end{align}
Therefore, we directly compute that both
$\partial_K(Z^* \wedge \tilde{\pi}_j^{p-1,q;n})$ and
$\partial_K(I^* \wedge \tilde{\pi}_j^{p,q;n})$ coincide with
$(n+1-(p+q)+j)(\partial_K I^*) \wedge \tilde{\pi}_j^{p,q;n}$, whereas
$$
2i\partial_K(I^* \wedge Z^* \wedge \overline{Z}^* \wedge
\tilde{\pi}_j^{p-1,q-1;n-1})
$$
equals $-(\partial_K I^*) \wedge \tilde{\pi}_{j+1}^{p,q;n}$. For the Poisson kernel,
this means that the coefficient of $(\partial_K I^*)\wedge \tilde{\pi}_j^{p,q;n}$ in
$\partial_K\phi_{p,q}$ equals $(n+1-(p+q)+j)(\alpha_j + \gamma_j) - \delta_{j-1}$,
which coincides with $-\beta(n+1)\kappa_{j+1}^{p+1,q;n}$ by a direct
computation. Hence we get
\begin{align*}
 \partial_K\phi_{p,q} = \beta \textstyle\sum_j  \left(\kappa_{j+1}^{p+1,q;n+1} \partial_K\left(\overline{Z}^* \wedge \tilde{\pi}_j^{p,q-1;n}\right) -(n+1)\kappa_{j+1}^{p+1,q;n} (\partial_K I^*) \wedge \tilde{\pi}_j^{p,q;n}\right).
\end{align*}
Our aim is to compare this with the $P$-derivative of $\phi_{p,q}$ for $p+q > n+1$. For this, recall from the proof of Theorem \ref{thm_main2} that $d_P(Z^* \wedge \tilde{\pi}_j^{p-1,q;n})$, $d_P(\overline{Z}^* \wedge \tilde{\pi}_j^{p,q-1;n})$ as well as $d_P(I^* \wedge \tilde{\pi}_j^{p,q;n})$ all coincide with $\tilde{\pi}_j^{p,q;n+1}$. Furthermore, for the $P$-derivative of $I^* \wedge Z^* \wedge \overline{Z}^* \wedge \tilde{\pi}_{j}^{p-1,q-1;n-1}$ we first apply the formulae for the derivatives of the basic invariant forms from Section \ref{4.4} and then \eqref{eq:K_derivative_t_pi} to write it as a linear combination of the elements $\overline{Z}^* \wedge \partial_K\tilde{\pi}_j^{p-1,q-1;n}$ and $(\partial_K I^*) \wedge \tilde{\pi}_j^{p-1,q;n}$. Now if we define $\tilde{\gamma}_j := \alpha_j + \beta_j + \gamma_j$ as in the proof of Theorem \ref{thm_main2} a direct computation shows that $\tfrac{\delta_j}{(n-(p+q)+j+2)} = -\tilde{\gamma}_{j+1}$. All in all, we obtain
\begin{align}
 d_P\phi_{p,q} = \textstyle\sum_j \tilde{\gamma}_j \tilde{\pi}_j^{p,q;n+1} + 2i \tilde{\gamma}_{j+1} \overline{Z}^* \wedge \partial_K\tilde{\pi}_{j+1}^{p-1,q-1;n} - 2i \delta_j (\partial_K I^*) \wedge \tilde{\pi}_j^{p-1,q;n}.
\end{align}
Next, we shift the index in the first summand and use that $\partial_K\overline{Z}^* = -\overline{\partial}_KZ^*$ to rewrite the first two summands as $\tilde{\gamma}_{j+1} \partial_K(\overline{Z}^* \wedge \tilde{\pi}_{j+1}^{p-1,q-1;n})$. Inserting the explicit expressions for $\tilde{\gamma}_j$ and $\delta_j$ from Theorem \ref{thm_main2} a direct computation yields
\begin{align}\label{eq:P_derivative_t_pi}
 d_P\phi_{p,q} = -2ic_{\alpha,\beta}\textstyle\sum_j \kappa^{p,q;n+1}_{j+1}  \partial_K\left(\overline{Z}^* \wedge\tilde{\pi}_{j+1}^{p-1,q-1;n}\right) - (n+1) \kappa_{j+1}^{p,q;n} (\partial_K I^*) \wedge \tilde{\pi}_j^{p-1,q;n}
\end{align}
with $c_{\alpha,\beta}= \tfrac{\alpha(n+1-q) + \beta(n+1-p)}{p+q-n}$. In particular, for $\alpha = 0$ we compare \eqref{eq:K_derivative_t_pi} with \eqref{eq:P_derivative_t_pi} and get
\begin{align}\label{eq:relation_t_pi_derivatives}
-2i(n+1-p) \partial_K \phi_{p-1,q}^{0,\beta} = (p+q-n) d_P\phi_{p,q}^{0, \beta}.
\end{align}
Furthermore, using $\overline{\phi_{p,q}^{\alpha,\beta}} = \phi_{q,p}^{\overline{\beta}, \overline{\alpha}}$ we obtain formulae for the operator $\overline{\partial}_K$ similar to \eqref{eq:K_derivative_t_pi} and \eqref{eq:relation_t_pi_derivatives}. 

Finally, we determine the image of $\phi_{p,q}$ under $\partial_K^*$. First, note that by \eqref{tpi-def} we can write $Z^* \wedge \tilde{\pi}_j^{p-1,q;n}$ as the sum of $2^n(Z^* \wedge \omega_j^{p-1,q;n})$ and a multiple of  $I^* \wedge Z^* \wedge \overline{Z}^* \wedge \omega_j^{p-2,q-1;n-1}$. Now by Proposition \ref{prop_formula_K_codifferential} the latter is contained in the kernel of $\partial_K^*$, implying that $\partial_K^*(Z^* \wedge \tilde{\pi}_j^{p-1,q;n})$ coincides with $2^n \partial_K^*(Z^* \wedge \omega_j^{p-1,q;n})$. In the same way, we can argue for the other summands of $\phi_{p,q}$, obtaining
\begin{align}\label{eq:tpi_K_codifferential_1}
 \partial_K^*\phi_{p,q} = 2^n\textstyle\sum_j \partial_K^*\left(\alpha_j Z^*\wedge \omega_j^{p-1,q;n} + \beta_j \overline{Z}^* \wedge \omega_j^{p,q-1;n} + \gamma_j I^* \wedge \omega_j^{p,q;n}\right).
\end{align}
In this form we can directly apply the formulae from Proposition \ref{prop_formula_K_codifferential} to \eqref{eq:tpi_K_codifferential_1}. After shifting the indices appropriately we have
\begin{align}
\begin{aligned}
 \partial_K^*\phi_{p,q} = 2^n\textstyle\sum_j &\alpha_j \omega_j^{p-1,q;n} + 2i b_j Z^* \wedge\overline{Z}^* \wedge \omega_j^{p-2,q-1;n-1} \\
 &+ 2i c_j I^* \wedge Z^* \wedge \omega_j^{p-2,q;n-1} + 2i d_j I^* \wedge \overline{Z}^* \wedge \omega_j^{p-1,q-1;n-1},
\end{aligned}
 \end{align}
where the coefficients are given by 
\begin{align*}
 b_j &:= (j+1)(n-(p+q)+j+2) \alpha_{j+1} - (p-j-1)(q-j)\alpha_j \\
 c_j &:= (j+1)(n-(p+q)+j+2) \alpha_{j+1} - (p-j-1)(q-j+1) \alpha_j \\
 d_j &:= (j+1)(n-(p+q)j+2) (\beta_{j+1} + \gamma_{j+1}) + (p-j)(q-j)(\beta_j + \gamma_j)
\end{align*}
Inserting the definitions of the coefficients $\alpha_j$, $\beta_j$ and $\gamma_j$ we can exploit Lemma \ref{lem_recursion_kappa} to simplify these expressions. In particular, a direct computation shows that $d_j$ does not depend on the parameter $\beta$. Explicitly, a direct computation shows that
\begin{align}\label{eq:tpi_K_codifferential}
 \begin{aligned}
  \partial_K^*\phi_{p,q} := 2\alpha\sum_j &\left(\kappa_{j+1}^{p,q+1;n+1} \omega_j^{p-1,q;n} - i(n+1)\kappa_{j+2}^{p,q+1;n} Z^* \wedge \overline{Z}^* \wedge \omega_j^{p-2,q-1;n-1}\right. \\
  &- i(q+2) \kappa_{j+2}^{p,q+2;n+1} I^* \wedge Z^* \wedge \omega_j^{p-2,q;n-1} \\
  &+ \left. i(p+1) \kappa_{j+2}^{p+1,q+1;n+1} I^* \wedge \overline{Z}^*  \wedge \omega_j^{p-1,q-1;n-1}\right).
 \end{aligned}
\end{align}
In particular, we deduce that
$ \partial_K^*\phi_{p,q}^{\alpha, \beta} = \partial_K^*\phi_{p,q}^{\alpha, 0}$, and
this vanishes if and only if $\alpha = 0$. Furthermore, since
$\overline{\phi_{p,q}^{\alpha, \beta}} =
\phi_{q,p}^{\overline{\beta},\overline{\alpha}}$
we easily obtain an analogous formula for $\overline{\partial}_K^*\phi_{p,q}$. In
particular, applying the complex conjugation to \eqref{eq:tpi_K_codifferential}
directly we deduce that
\begin{align}\label{eq:relation_t_pi_codifferential}
 \partial_K^* \phi_{p+1,q}^{\alpha, 0} = \overline{\partial}_K^* \phi_{p,q+1}^{0, \alpha}
\end{align}
for all $p+q \ge n$.
  \end{proof}

\subsection{Real versions of Poisson transforms}
In the final section we construct for all $0\le k \le 2n$ Poisson transforms $\Phi_k \colon \Omega^k(G/P) \to \Omega^k(G/K)$ which factor to the Rumin complex. Their associated kernels will be defined as special linear combinations of the operators from Theorems \ref{thm_main1} and \ref{thm_main2}, respectively. Explicitly, recall from Proposition \ref{prop_properties_below_half} that for $p+q \le n$ the image of $\Phi_{p,q}$ consist of coclosed differential forms, whereas from Proposition \ref{prop_properties_above_half} that for $p+q > n$ the the $2$-parameter family $\Phi_{p,q}^{\alpha,\beta}$ satisfies relations between its compositions with the partial derivatives $\partial$ and $\overline{\partial}$ on $G/K$ and the exterior derivative on $G/P$.

Therefore, we will design the real operators $\Phi_k$ so that their image consists of coclosed differential forms and that they satisfy a commutation relation with respect to the exterior derivatives on $G/K$ and $G/P$. Explicitly, for all $0 \le k \le n$ we define the real differential form $\varphi_k := \sum_{p+q = k} \lambda_{p,q} \phi_{p,q}$ with coefficients
\begin{align*}
 \lambda_{p,q} := 2^{-k} i^{q-p} (n+1-p)!(n+1-q)!
\end{align*}
and for all $n+1 \le k \le 2n$ the form $\phi_{k} := \sum_{p+q=k} \phi_{p,q}^{\alpha_{p,q}, \beta_{p,q}}$ with parameters
\begin{align*}
 \alpha_{p,q} &:= 2^{-k-1} i^{q-p} (n+1-p)!(n-q)!, & \beta_{p,q} &:= 2^{-k-1} i^{q-p} (n-p)!(n-q+1)!.
\end{align*}
These Poisson kernels satisfy $\overline{\phi_k} = \phi_k$ for all $0 \le k \le 2n$ and therefore induce Poisson transforms $\Phi_k \colon \Omega^k(G/P) \to \Omega^k(G/K)$
which are $G$-equivariant and factor to the Rumin complex.

\begin{thm}
 For all $0 \le k \le 2n$ let $\underline{\Phi}_k \colon \Gamma(\mathcal{H}_k(G/P)) \to \Omega^k(G/K)$ be the $G$-equivariant linear operators induced by $\Phi_k$. Then these operators satisfy:
 \begin{enumerate}
  \item[(i)] The image of $\underline{\Phi}_k$ is contained in the space of harmonic and coclosed differential forms, which are primitive for $0 \le k \le n$ and coprimitive for $n+1 \le k \le 2n+1$.
  \item[(ii)] If $D_k \colon \Gamma(\mathcal{H}_k) \to \Gamma(\mathcal{H}_{k+1})$ denotes the $k$th BGG-operator, then
\begin{align*}
  d \circ \underline{\Phi}_k = c_k \ \underline{\Phi}_{k+1} \circ D_k
 \end{align*}
 for all $0 \le k \le 2n$, where the coefficient $c_k$ is defined by
 \begin{align*}
  c_k = \begin{cases} n-k+1 & k \le n \\
         n-k-1 & k \ge n+1
        \end{cases}
 \end{align*}

 \end{enumerate}
\end{thm}

\begin{proof}
  (i) The operators $\underline{\Phi}_k$ are defined by
  $\underline{\Phi}_k(\pi(\alpha)) = \Phi_k(\alpha)$ for all
  $\alpha \in \Gamma(\ker(\partial^*))$, where
  $\pi \colon \Gamma(\ker(\partial^*))\to\Gamma(\mathcal{H}_k)$ denotes the canonical
  projection. Therefore, it suffices to prove the claimed properties of the image of
  $\Phi_k$. But from Proposition \ref{prop_properties_below_half} we already know
  that for $p+q \le n$ the image of the operators $\Phi_{p,q}$ consists of harmonic,
  coclosed and primitive differential forms, so by linearity the same is true for the
  images of $\Phi_k$ for $0 \le k \le n$. Moreover, for $p+q > n$ we have shown in
  Proposition \ref{prop_properties_above_half} that the image of $\Phi_{p,q}$ is
  contained in the space of harmonic and coprimitive differential forms, and thus the
  same holds for $\Phi_{k}$ for $n+1 \le k \le 2n$. Hence it remains to show that
  $\delta \circ \Phi_k = 0$ for all $n+1 \le k \le 2n+1$, which by Proposition
  \ref{prop3.2} amounts to proving that $\phi_k$ is $K$-coclosed. But we have seen in
  Proposition \ref{prop_properties_above_half} that
  $\partial_K\phi^{\alpha,\beta}_{p,q} = \partial_K\phi^{\alpha, 0}_{p,q}$ for all
  parameter $\alpha$, $\beta \in \mathbb{C}$, and similarly for
  $\overline{\partial}_K^*$. Therefore, using formula
  \eqref{eq:relation_t_pi_codifferential} we deduce that
\begin{align*}
 \delta_K\phi_k = \sum_{p+q=k} \partial_K^*\phi_{p,q}^{\alpha_{p,q},0} + \overline{\partial}_K^*\phi_{p,q}^{0,\beta_{p,q}} = \sum_{p+q=k} \partial_K^* \phi_{p,q}^{\gamma_{p,q},0}
\end{align*}
with parameter $\gamma_{p,q} := \alpha_{p,q} + \beta_{p-1,q+1}$. But a direct computation shows that $\gamma_{p,q} = 0$ and hence $\delta_K\phi_k = 0$.
 
(ii) Using the definition of $\underline{\Phi}_k$ and Proposition \ref{prop2.4} we
need to show that
 \begin{align}\label{eq:relation_real_partial_derivatives}
  d_K\phi_k = c_k d_P\phi_{k+1}
 \end{align}
for all $0 \le k \le 2n$. 

First, for $p+q \le n$ the Poisson kernel $\phi_{p,q}$ was given by
$ I^* \wedge \sum_{j}\kappa_j^{p,q;n+1} \pi_j^{p,q;n}$ with the forms $\pi_j^{p,q;n}$
being defined in \eqref{pi-def1}. Using the formulae for the derivatives of the basic
invariant forms from Section \ref{3.4} a direct computation shows that
$d_K(I^* \wedge \pi_j^{p,q;k}) = (k+1-(p+q)) (Z^* + \overline{Z}^*) \wedge
\pi_j^{p,q;k}$ and therefore
\begin{align}\label{eq:K_derivative_pi}
 d_K\phi_{p,q} = (n+1-(p+q)) (Z^* + \overline{Z}^*) \wedge \phi_{p,q}.
\end{align}
Similarly, the definition of $\pi_j^{p,q;n}$ readily implies that
$d_P(I^* \wedge \pi_j^{p,q;n}) = \pi_j^{p,q;n+1}$, which we write as a linear
combination of $\omega_j^{p,q;n+1}$, $I^* \wedge Z^* \wedge \omega_j^{p,q;n}$ and
$I^* \wedge \overline{Z}^* \wedge \omega_j^{p,q;n}$. Inserting this into
$d_P\phi_{p,q}$ the sum $\sum_j\kappa_j^{p,q;n+1}\omega_j^{p,q;n+1}$ vanishes due to
Proposition \ref{prop_relation_omega}. For the other summands, a direct computation
using the definition of the coefficients $\kappa$ implies that
\begin{align}\label{eq:P_derivative_pi}
 d_P\phi_{p,q} = 2i \left((n+2-p) Z^* \wedge \phi_{p-1,q} - (n+2-q) \overline{Z}^* \wedge \phi_{p,q-1}\right).
\end{align}
In order to compute $d_K\phi_k$ in the case $k < n$ we first use linearity of the
$K$-differential and apply \eqref{eq:K_derivative_pi} to each of the summands of
$\phi_k$. Next, we apply the relation $\lambda_{p,q} = 2i\lambda_{p+1,q}(n+1-p)$ to
the coefficients of the summands $Z^* \wedge \phi_{p,q}$ and
$\lambda_{p,q} = -2i\lambda_{p,q+1}(n+1-q)$ to the coefficients of
$\overline{Z}^* \wedge \phi_{p,q}$. Shifting the summation indices of the resulting
expression to $p+q = k+1$ and comparing this to \eqref{eq:P_derivative_pi} yields
\eqref{eq:relation_real_partial_derivatives} for all $0 \le k < n$.

Next, if $p+q > n$ we have seen in the proof of Proposition
\ref{prop_properties_above_half} that
$\partial_K\phi_{p,q}^{\alpha,\beta}= \partial_K\phi_{p,q}^{0, \beta}$ and
$\overline{\partial}_K \phi_{p,q}^{\alpha,\beta} = \overline{\partial}_K
\phi_{p,q}^{\alpha, 0}$
for all $\alpha$, $\beta \in \mathbb{C}$. Since $p+q > n$ we know that both $p$ and
$q$ are positive, so linearity of the $K$-differential and shifting the indices
appropriately we obtain
\begin{align}\label{eq:relation_real_partial_derivatives_above_half}
d_K\phi_k = \sum_{p+q=k} \overline{\partial_K} \phi_{p,q}^{\alpha_{p,q}, 0}
  + \partial_K \phi_{p,q}^{0,\beta_{p,q}} = \sum_{p+q = k+1}
  \overline{\partial}_K\phi_{p,q-1}^{\alpha_{p,q-1},0} + \partial_K \phi_{p-1,q}^{0,
  \beta_{p-1,q}}. 
\end{align}
By definition, the parameters of $\phi_k$ satisfy
$\alpha_{p,q-1} = -2i(n+1-q)\alpha_{p,q}$ as well as
$\beta_{p-1,q} = 2i(n+1-p)\beta_{p,q}$. Inserting this into
\eqref{eq:relation_real_partial_derivatives_above_half} and afterwards applying
formula \eqref{eq:relation_t_pi_derivatives} for the relation between the partial
derivatives of $\phi_{p,q}^{\alpha,\beta}$ a direct computation yields
\eqref{eq:relation_real_partial_derivatives} for all $n < k \le 2n$.

Finally, we need to show the formula in the special case $k = n$. In order to do so,
using the definition of the forms $\tilde{\pi}_j$ a similar computation as in the
proof of Proposition \ref{prop_properties_above_half} show that the derivative
$d_P\phi_{p,q}$ for $p+q = n+1$ can be written as
\begin{align}\label{eq:P_derivative_half_dimension}
d_P\phi_{p,q} = \sum_j \tilde{\gamma}_j \tilde{\pi}_j^{p,q;n+1} + 2i \delta_j d_P(I^* \wedge Z^* \wedge \overline{Z}^*) \wedge \tilde{\pi}_j^{p-1,q-1;n-1},
\end{align}
where the coefficients $\tilde{\gamma}_j$ and $\delta_j$ are given by
$\tilde{\gamma}_j = (\alpha p + \beta q) \kappa_j^{p,q;n+1}$ and
$\delta_j = -(n+1)(\alpha p + \beta q) \kappa_{j+1}^{p,q;n}$. Next, we expand the
form $\tilde{\pi}_j^{p,q;n+1}$ in terms of wedge products of the $\omega$'s with
invariant $1$-forms. Inserting this into \eqref{eq:P_derivative_half_dimension} we
obtain the summand $\sum_j \tilde{\gamma}_j\omega_j^{p,q;n+1}$, which is trivial due
to Proposition \ref{prop_relation_omega}. Therefore, by definition of the forms
$\pi_j^{p,q;k}$ we can replace $\tilde{\pi}_j^{p,q;n+1}$ in
\eqref{eq:P_derivative_half_dimension} by the expression
\begin{align*}
2ij Z^* \wedge \overline{Z}^* \wedge \pi_{j-1}^{p-1,q-1;n} + 2i(p-j) Z^* \wedge I^* \wedge \pi_j^{p-1,q;n} - 2i(q-j) \overline{Z}^* \wedge I^* \wedge \pi_j^{p,q-1;n},
\end{align*}
and expanding the other summand as well we deduce that
\begin{align}
 d_P\phi_{p,q} = 2i\sum_j \epsilon_j Z^* \wedge \overline{Z}^* \wedge \pi_{j-1}^{p-1,q-1;n} + \zeta_j Z^* \wedge I^* \wedge \pi_j^{p-1,q;n} - \eta_j \overline{Z}^* \wedge I^* \wedge \pi_j^{p,q-1;n},
\end{align}
with coefficients $\epsilon_j = (j+1)\tilde{\gamma}_{j+1} + \delta_j$, $\zeta_j = (p-j)\tilde{\gamma}_j - \delta_j $ and $\eta_j = (q-j)\tilde{\gamma}_j - \delta_j$. Inserting the definitions we obtain by a direct computation that $\epsilon_j = 0$, whereas $\zeta_j = (\alpha p+ \beta q)(q+1) \kappa_{j+1}^{p,q+1;n+1}$ and $\eta_j = (\alpha p + \beta q) (p+1) \kappa_{j+1}^{p+1,q;n+1}$. 

In order to determine the $P$-derivative of $\phi_{n+1}$ we use linearity of $d_P$ and apply the formula \eqref{eq:P_derivative_half_dimension}. Next, we shift the  summation to all $p,q$ with $p+q = n$, so that we can write $d_P\phi_{n+1}$ as a linear combination of $Z^* \wedge I^* \wedge \pi_j^{p,q;n}$ and $\overline{Z}^* \wedge I* \wedge \pi_j^{p,q;n}$. Inserting the definition of the parameter $\alpha$ and $\beta$ and using that $\lambda_{p,q} = 2^{-n} i^{q-p}(p+1)!(q+1)!$ for $p+q = n$ a direct computation yields
\begin{align*}
 d_P\phi_{n+1} &= 2i\sum_{p+q=n+1} \sum_j \zeta_j Z^* \wedge I^* \wedge \pi_j^{p-1,q;n} - \eta_j \overline{Z}^* \wedge I^* \wedge \omega_j^{p,q-1;n} \\
 &= \sum_{p+q=n} \lambda_{p,q}\kappa_{j+1}^{p+1,q+1;n+1} (Z^* + \overline{Z}^*)\wedge I^* \wedge \pi_j^{p,q;n}.
\end{align*}
Finally, since $p+q = n$ we conclude that $\kappa_{j+1}^{p+1,q+1;n+1} = \kappa_j^{p,q;n+1}$ and thus
\begin{align*}
 d_P\phi_{n+1} = \sum_{p+q = n} \lambda_{p,q} (Z^* + \overline{Z^*})\wedge \phi_{p,q},
\end{align*}
which coincides with $d_K\phi_n$ due to \eqref{eq:K_derivative_pi}. 
\end{proof}

\appendix
\section{Some technical computations} 
In this appendix we derive explicit formulae for the image of the Poisson kernels
$\omega_j^{p,q;k}$ defined in Section \ref{4.2} under the $K$-Hodge star $\ast_K$,
the $K$-codifferential $\delta_K$ and the map $\Cal L_K^*$ corresponding to the
adjoint of the Lefschetz map on $G/K$.

\subsection{The $K$-Hodge star}
Recall from Section \ref{4.2} that for all positive integers $p$, $q$ and $k$ with
$0 \le p, q, k-p, k-q \le n$, the $M$-invariant elements $\omega_j^{p,q;k}$ are
defined by
\begin{align*}
 \omega_j^{p,q;k} = \omega_{2,0}^{j} \wedge \omega_{1,1}^{p-j} \wedge \overline{\omega_{1,1}}^{q-j} \wedge \omega_{0,2}^{k-(p+q)+j}.
\end{align*}
In order to obtain formulae for $\delta_K$ and $\Cal L_K^*$ it will be necessary to
write expressions of the form $\omega \wedge \ast_K\omega_j^{p,q;k}$ again as an
image of $\ast_K$, where $\omega$ is any of the $M$-invariant elements in
$\Lambda^* (\frak g/\frak m)^*$ of degree $2$.

We will start by identifying the pullback $\omega_K$ of the K\"{a}hler form on $G/K$
along the canonical projection $\pi_K \colon G/M \to G/K$. By construction this is a
real and non-degenerate $M$-invariant form on $(\frak g/\frak m)_{\Bbb C}$ of degree
$(2,0)$ and $K$-type $(1,1)$ and therefore given by a multiple of
$iZ^* \wedge \overline{Z}^* + \omega_{2,0}$. Furthermore, using the relation
$\omega_K(J\xi, \eta) = g_K(\xi, \eta)$, where $g_K$ is the pullback of the
$K$-invariant Hermitian inner product on $\frak g/ \frak k$, we deduce that
\begin{align}\label{eq:pullback_Kaehler_form}
 \omega_K = \frac{1}{2}(iZ^* \wedge \overline{Z}^* + \omega_{2,0}).
\end{align}
From this description we obtain that the volume form $\vol_K := \ast_K(1)$ on the
horizontal subspace $\mathfrak{p}/\mathfrak{m}$ of $\mathfrak{g}/\mathfrak{m}$ is
given by 
\begin{align}\label{eq:volume_form}
 \vol_K = \frac{1}{2^{n+1}n!} iZ^* \wedge \overline{Z}^* \wedge \omega_{2,0}^{n}.
\end{align}

In the next step we determine a formula which expresses the $M$-invariant forms
$\ast_K\omega_j^{p,q;k}$ in terms of powers of the invariant $2$-form
$\omega_{2,0}$. In order to do so, note that $\omega_{0,2}$ is of degree $(0,2)$ and
thus commutes with the $K$-Hodge star by construction. Thus, it suffices to consider
those $M$-invariant elements $\ast_K\omega_j^{p,q;k}$ which do not contain
$\omega_{0,2}$ as a factor. By definition, this means that $j = p+q-k \ge 0$, and for
any choice of $p$, $q$ and $k$ there is precisely one such $M$-invariant form.

For simplification of the following proofs we introduce some notation. By a
\emph{holomorphic $k$-vector} ${\bf X}^{1,0}$ we mean a $k$-tuple
${\bf X}^{1,0} := (X_1^{1,0}, \dotsc, X_k^{1,0})$ whose entries are the holomorphic
parts of vectors $X_1,$ $\dotsc,$ $X_k \in \mathbb{C}^n$. Next, we denote by
$F_{\bf X}^{1,0}$ the $k$-tuple $(F_{X^{1,0}_1}, \dotsc, F_{X_k^{1,0}})$ and
abbreviate by $\iota_{F_{\bf X}^{1,0}}$ the interior product
$\iota_{F_{X_1}^{1,0}} \dotso \iota_{F_{X_k}^{1,0}}$. In a similar way we define
\emph{antiholomorphic $k$-vectors} ${\bf Y}^{0,1}$ and the corresponding tuples
$F_{\bf Y}^{0,1}$ and interior products $\iota_{F_{\bf Y}^{0,1}}$. Finally, we also
define $G_{\bf X}^{1,0}$ and $G_{\bf Y}^{0,1}$ as well as their interior products in
a similar fashion.

\begin{lemma}\label{lem_power_omega_2_0}
  Let $p$, $q$ and $k$ be integers so that $0 \le p, q, k-p, k-q \le n$. Let
  $\textbf{X}^{1,0}$ be a holomorphic $(k-p)$-vector and $\textbf{Y}^{0,1}$ an
  antiholomorphic $(k-q)$-vector, respectively. Then
 \begin{equation}\label{eq:power_omega_2_0}
   \iota_{G_{\bf X}^{1,0}}\iota_{G_{\bf Y}^{0,1}}\ast_K\omega_{p+q-k}^{p,q;k} =  \varepsilon^{p,q;k}iZ^* \wedge \overline{Z}^* \wedge \iota_{F_\textbf{X}^{1,0}}\iota_{F_\textbf{Y}^{0,1}}\omega_{2,0}^{n-(p+q)+k}, 
 \end{equation}
where the factor $\varepsilon^{p,q;k}$ is defined by
\begin{align*}
  \varepsilon^{p,q;k} = (-1)^{\frac{1}{2}(p+q)(p+q+1)-(p+q-k)}2^{p+q-(n+1)}\tfrac{(p+q-k)!(k-p)!(k-q)!}{(n-(p+q)+k)!} .
\end{align*}
\end{lemma}
\begin{proof}
  We will prove this formula by induction on the $P$-type of
  $\ast_K\omega_{p+q-k}^{p,q;k}$. If the $P$-type equals $(0,0)$, then $k = p = q$,
  so there are no interior products in \eqref{eq:power_omega_2_0}. From the explicit
  description of the Hermitian product $g_K$ we compute
  $g_K(\omega_{2,0}^{k}, \omega_{2,0}^k) = 2^{2k} \frac{n! k!}{(n-k)!}$, which
  together with expression \eqref{eq:volume_form} for the volume form yields
\begin{align*}
  \ast_K \omega_k^{k,k;k} = \ast_K\omega_{2,0}^k =  2^{2k-(n+1)}\tfrac{k!}{(n-k)!}
  iZ^* \wedge \overline{Z}^* \wedge \omega_{2,0}^{n-k}. 
\end{align*}
Note that the overall factor on the right hand side coincides with
$\varepsilon^{k,k;k}$, which completes the argument for $P$-type $(0,0)$.
 
Now assume that \eqref{eq:power_omega_2_0} is shown for all $M$-invariant elements
$\ast_K\omega_{p+q-k}^{p,q;k}$ of $P$-type $(k-p,k-q)$ and consider
$\ast_K\omega_{p+q-k}^{p,q+1;k+1} = \ast_K (\omega_{p+q-k}^{p,q;k} \wedge
\overline{\omega_{1,1}})$,
which by definition is of $P$-type $(k-p+1, k-q)$. Using that
$\iota_{G_X^{1,0}}\overline{\omega_{1,1}} = (F_X^{1,0})^{\flat}$ for all
$X \in \Bbb C^n$ a direct computation yields
\begin{align*}
 \iota_{G_{X}^{1,0}}\ast_K\!\omega_{p+q-k}^{p,q+1;k+1} &= (-1)^{p+q+1} 2(k-p+1) \iota_{F_{X}^{1,0}}\ast_K\!\omega_{p+q-k}^{p,q;k}.
\end{align*}
 We apply this equation to \eqref{eq:power_omega_2_0} for $X = X_1$ and afterwards move the resulting interior product with $F_{X_1}^{1,0}$ to the left. On the remaining term we insert the induction hypothesis and finally use the relation
 \begin{align*}
 (-1)^{p+q+1}2(k-p+1)\varepsilon^{p,q;k} = \varepsilon^{p,q+1;k+1},
 \end{align*}
thereby obtaining the claim. Since our formula is symmetric in the parameters $p$ and $q$ the same line of arguments can also be applied to  $\ast_K\omega_{p+q-k}^{p+1,q;k+1}$, which is of $P$-degree $(k-p, k-q+1)$.
\end{proof}

Using Lemma \ref{lem_power_omega_2_0} we are now able to express the wedge product of an $M$-invariant $2$-form and $\ast_K\omega_j^{p,q;k}$ again as an image of the $K$-Hodge star operator. This plays an essential role in the derivation of explicit formulae for the $K$-codifferential $\delta_K$ and the operator ${\Cal L}_K^*$ in the following section. 

\begin{thm}\label{thm_wedge_product_2_form_omega}
Let $p$, $q$ and $k$ be integers so that $0 \le p, q, k-p, k-q \le n$. Then:
 \begin{enumerate}
 \item[(i)] $\omega_{1,1} \wedge \ast_K \omega_{j}^{p,q;k} = (-1)^{p+q}2i\left(j\ast_K\!\omega_{j-1}^{p,q-1;k}+ (q-j)\ast_K\!\omega_j^{p,q-1;k}\right)$.
 \item[(ii)]$\overline{\omega_{1,1}} \wedge \ast_K\omega_{j}^{p,q;k} = (-1)^{p+q+1}2i\left(j\ast_K\!\omega_{j-1}^{p-1,q;k} + (p-j)\ast_K\!\omega_{j}^{p-1,q;k}\right)$.
 \item[(iii)] $\begin{aligned}[t]\omega_{2,0} \wedge \ast_K\omega_{j}^{p,q;k} =\ &4j(n+1-(p+q)+j)\ast_K\!\omega_{j-1}^{p-1,q-1;k-1} \\
 &- 4(p-j)(q-j) \ast_K\!\omega_j^{p-1,q-1;k-1}\end{aligned}$.
 \end{enumerate}
\end{thm}
\begin{proof} Since we can move the $(0,2)$-form $\omega_{0,2}$ outside of the $K$-Hodge star it suffices to consider those forms $\omega_j^{p,q;k}$ in which $\omega_{0,2}$ does not appear as a factor. By definition, this corresponds to $j = p+q-k$ for which Lemma \ref{lem_power_omega_2_0} is applicable. 

For part (i), let $\textbf{X}^{1,0}$ and $\textbf{Z}^{1,0}$ be holomorphic vectors on $\Bbb C^n$ of length $n-q+1$ and $k-p$, respectively, and let $\textbf{Y}^{0,1}$ and $\textbf{W}^{0,1}$ be antiholomorphic vectors of length $n-p$ and $k-q+1$, respectively. Then we rewrite the expression
\begin{equation}\label{eq:proof_theorem_a2_1}
\big(\omega_{1,1} \wedge \ast_K\omega_{p+q-k}^{p,q;k}\big)(Z, \overline{Z}, F_\textbf{X}^{1,0}, F_\textbf{Y}^{0,1}, G_\textbf{Z}^{1,0}, G_\textbf{W}^{0,1})
\end{equation}
as follows. First, we expand the wedge product by inserting the vectors $F_{X_s}^{1,0}$ and $G_{W_t}^{0,1}$ into $\omega_{1,1}$, where we pick up the signs $(-1)^{s+1}$ from the first vector and $(-1)^{t-k+q+1}$ from the second vector. Next, we use the relation $\omega_{1,1}(F_{X_s}^{1,0}, G_{W_t}^{0,1}) = -i\omega_{2,0}(F_{X_s}^{1,0}, F_{W_t}^{0,1})$ and apply Lemma \ref{lem_power_omega_2_0} to $\ast_K\omega_{p+q-k}^{p,q;k}$, which adds the coefficient $i\varepsilon^{p,q;k}$. All in all we obtain that \eqref{eq:proof_theorem_a2_1} equals
 \begin{align*}
   \textstyle\sum_{s,t}(-1)^{k-q+s+t}\varepsilon^{p,q;k}\omega_{2,0}(F_{X_s}^{1,0}, F_{W_t}^{0,1})\omega_{2,0}^{n-(p+q)+k}\left(F_\textbf{X}^{1,0;s}, F_\textbf{Y}^{0,1}, F_\textbf{Z}^{1,0}, F_\textbf{W}^{0,1;t}\right),
\end{align*}
where the wedge product $F_\textbf{X}^{1,0;s}$ of length $(n-q)$ is obtained from $F_\textbf{X}^{1,0}$ by omitting the $s$th factor, and similarly for $F_\textbf{W}^{0,1;t}$. Using that the Poisson kernel $\omega_{2,0}$ is of $K$-type $(1,1)$, the expression \eqref{eq:proof_theorem_a2_1} splits into the sum $(A) + (B)$, where 
\begin{align*}
 (A) &:= \tfrac{k-q+1}{n-(p+q)+k+1}\varepsilon^{p,q;k}\omega_{2,0}^{n-(p+q)+k+1}(F_\textbf{X}^{1,0}, F_\textbf{Y}^{0,1}, F_\textbf{Z}^{1,0}, F_\textbf{W}^{0,1}) \\
 (B) &:= \sum_{r = 1}^{k-p} \sum_{t=1}^{k-q+1} (-1)^{k-p+r+t}\varepsilon^{p,q;k}\omega_{2,0}(F_{Z_r}^{1,0}, F_{W_t}^{0,1})\omega_{2,0}^{n-(p+q)+k}(F_\textbf{X}^{1,0}, F_\textbf{Y}^{0,1}, F_\textbf{Z}^{1,0;r}, F_\textbf{W}^{0,1;t}).  
\end{align*}
Since $(A)$ is given by a power of $\omega_{2,0}$ we can immediately apply Lemma \ref{lem_power_omega_2_0}, and a direct computation yields
\begin{align*}
 (A) = (-1)^{p+q} 2i(p+q-k)\left(\ast_K\!\omega_{p+q-k-1}^{p,q-1;k}\right)(Z, \overline{Z}, F_\textbf{X}^{1,0}, F_\textbf{Y}^{0,1}, G_\textbf{Z}^{1,0}, G_\textbf{W}^{0,1}).
\end{align*}
In order to simplify $(B)$ we first apply Lemma \ref{lem_power_omega_2_0} to replace
$\omega_{2,0}^{n-(p+q)+k}$ with the invariant form $\ast_K\omega^{p,q-1;k-1}_{p+q-k}$
as well as change $\omega_{2,0}(F_{Z_r}^{1,0}, F_{W_t}^{0,1})$ to
$\omega_{0,2}(G_{Z_r}^{1,0}, G_{W_t}^{0,1})$. Finally, since the Poisson kernel
$\omega_{0,2}$ is of degree $(0,2)$ it can be moved into the $K$-Hodge star. All in
all, a direct computation yields
\begin{align*}
 (B) = (-1)^{p+q}2i(k-p) \left(\ast_K\omega_{p+q-k}^{p,q-1;k}\right)(Z, \overline{Z}, F_\textbf{X}^{1,0}, F_\textbf{Y}^{0,1}, G_\textbf{Z}^{1,0}, G_\textbf{W}^{0,1}).
\end{align*}
Since both expressions $(A)$ and $(B)$ have the same arguments as
\eqref{eq:proof_theorem_a2_1}, the corresponding $M$-invariant forms have to
coincide, proving (i). Furthermore, if we exchange $p$ and $q$ and apply complex
conjugation we immediately obtain the claimed formula in (ii).

Finally, to prove (iii) let $\textbf{X}^{1,0}$ and $\textbf{Z}^{1,0}$ be holomorphic
vectors on $\Bbb C^n$ of length $n-q+1$ and $k-p$, respectively, and
$\textbf{Y}^{0,1}$ and $\textbf{W}^{0,1}$ be antiholomorphic vectors on $\Bbb C^n$ of
length $n-p+1$ and $k-q$, respectively. We proceed similar as in the proof of (i) and
rewrite
\begin{equation}\label{eq:proof_theorem_a2_2}
 \left(\omega_{2,0} \wedge \ast_K\omega_{p+q-k}^{p,q;k}\right)(Z, \overline{Z},F_\textbf{X}^{1,0}, F_\textbf{Y}^{0,1}, G_\textbf{Z}^{1,0}, G_\textbf{W}^{0,1})
\end{equation}
as follows. First, we expand the wedge product by inserting the vectors
$F_{X_r}^{1,0}$ and $F_{Y_s}^{0,1}$ into $\omega_{2,0}$, which adds the sign
$(-1)^{n-q+r+s}$. Next, we apply Lemma \ref{lem_power_omega_2_0}, thereby changing
$\ast_K\omega^{p,q;k}_{p+q-k}$ to a multiple of $\omega_{2,0}^{n-(p+q)-k}$. All in
all we obtain that \eqref{eq:proof_theorem_a2_2} coincides with
\begin{align*}
 \textstyle\sum_{r,s} (-1)^{n-q+r+s} i\varepsilon^{p,q;k} \omega_{2,0}(F_{X_r}^{1,0}, F_{Y_s}^{0,1}) \omega_{2,0}^{n-(p+q)-k}\left(F_\textbf{X}^{1,0;r}, F_\textbf{Y}^{0,1;s}, G_\textbf{Z}^{1,0}, G_\textbf{W}^{0,1}\right),
\end{align*}
which can be written as $(C) + (D)$, where
\begin{align*}
 (C) &:= \tfrac{n-q+1}{n-(p+q)+k+1} i\varepsilon^{p,q;k} \omega_{2,0}^{n-(p+q)+k+1}(F_\textbf{X}^{1,0}, F_\textbf{Y}^{0,1}, F_\textbf{Z}^{1,0}, F_\textbf{W}^{0,1}),\\
 (D) &:= \textstyle\sum_{r, t}i\varepsilon^{p,q;k} (-1)^{k-q+r+t+1} \omega_{2,0}(F_{X_r}^{1,0}, F_{W_t}^{0,1}) \omega_{2,0}^{n-(p+q)+k}(F_\textbf{X}^{1,0;r}, F_\textbf{Y}^{0,1}, F_\textbf{Z}^{1,0}, F_\textbf{W}^{0,1;t}).
\end{align*}
For $(C)$ we immediately apply Lemma \ref{lem_power_omega_2_0}, which by a direct computation yields
\begin{align*}
 (C) = 4(n-q+1)(p+q-k)\ast_K\omega_{p+q-k-1}^{p-1,q-1;k-1}(Z, \overline{Z}, F_{\bf X}^{1,0}, F_{\bf Y}^{0,1}, G_{\bf Z}^{1,0}, G_{\bf W}^{0,1}).
\end{align*}
On the other hand, the expression $(D)$ has the same form as \eqref{eq:proof_theorem_a2_1}, hence by a direct computation we obtain
\begin{align*}
 (D) = (-1)^{p+q+1}2i(k-q) \left(\omega_{1,1} \wedge \ast_K\!\omega_{p+q-k}^{p-1,q;k-1}\right)(Z, \overline{Z}, F_\textbf{X}^{1,0}, F_\textbf{Y}^{0,1}, G_\textbf{Z}^{1,0}, G_\textbf{W}^{0,1}).
\end{align*}
All in all we have
\begin{align*}
 \omega_{2,0} \wedge \ast_K\omega_{p+q-k}^{p,q;k} =\ &4(n-q+1)(p+q-k)\ast_K\omega_{p+q-k-1}^{p-1,q-1;k-1} \\
 &+ (-1)^{p+q+1}2i(k-q) \omega_{1,1} \wedge \ast_K\omega_{p+q-k}^{p-1,q;k-1},
\end{align*}
so by applying formula (i) the claim follows.
\end{proof}

\subsection{Formulae for $\delta_K$ and $\mathcal{L}_K^*$}
Following the previous section we now determine explicit formulae for the images of
the $M$-invariant forms $\omega_j^{p,q;k}$ under the $M$-equivariant maps $\delta_K$
and $\mathcal{L}_K^*$.

Recall from Section \ref{3.2} that the decomposition of Poisson kernels into
$K$-types induces a splitting $d_K = \partial_K + \overline{\partial}_K$, where the
first and second operator map Poisson kernels of $K$-type $(p,q)$ to those of
$K$-type $(p+1,q)$ and $(p, q+1)$, respectively.  Indeed, the analogous arguments as
in Proposition \ref{prop3.2} show that if $\Phi$ is any Poisson transform with kernel
$\phi$, then $\partial \circ \Phi$ is again a Poisson transform with kernel
$\partial_K\phi$, and similarly for its complex conjugate. Furthermore, the operators
$\partial_K$ and $\overline{\partial_K}$ are again antiderivations, square to zero,
anticommute with each other and are related via
$\overline{\partial}_K\overline{\phi} = \overline{\partial_K\phi}$ for all
$M$-invariant forms $\phi$ on $(\frak g/\frak m)_{\Bbb C}$.

Similarly, the $K$-codifferential decomposes as
$\delta_K = \partial_K^* + \overline{\partial}^*_K$, where the first and second
operator map Poisson kernels of $K$-type $(p,q)$ to those of $K$-type $(p-1,q)$ and
$(p,q-1)$, respectively. Explicitly, we have
$\partial_K^* = -\ast_K \overline{\partial}_K \ast_K$ and similarly for
$\overline{\partial}_K^*$, so since the $K$-Hodge star is complex linear they are
related via $ \overline{\partial}^*_K\overline{\phi} = \overline{\partial_K^*\phi}$
for all Poisson kernels $\phi$.

\begin{prop}\label{prop_formula_K_codifferential}
 Let $p$, $q$ and $k$ be integers so that $0 \le p, q, k-p, k-q \le n$. Then we have
  \begin{align*}
  \partial^*_K\big(Z^* \wedge \overline{Z}^* \wedge \omega_j^{p,q;k}\big) =\ &2(n-k) \overline{Z}^* \wedge \omega_j^{p,q;k}\\
 &+ 2ij(k-(p+q)+j)  I^* \wedge Z^* \wedge \overline{Z}^* \wedge \omega_{j-1}^{p-1,q;k-1} \\
 &-2i(p-j)(n-k+p-j+1) I^* \wedge Z^* \wedge \overline{Z}^* \wedge \omega_{j}^{p-1,q;k-1}.
 \end{align*}
 Furthermore, this implies
 \begin{align*}
 \partial_K^*(Z^* \wedge \omega_j^{p,q;k}) =\ &2(n+1-k)\omega_j^{p,q;k} \\
 &+ 2ij(n+1-(p+q)+j) Z^* \wedge \overline{Z}^* \wedge \omega_{j-1}^{p-1,q-1;k-1} \\
 &-  2i(p-j)(q-j) Z^* \wedge \overline{Z}^* \wedge  \omega_j^{p-1,q-1;k-1} \\
 &+ 2ij(k-(p+q)+j) I^* \wedge Z^* \wedge \omega_{j-1}^{p-1,q;k-1} \\
 &- 2i(p-j)(n-k+q-j+1) I^* \wedge Z^* \wedge \omega_j^{p-1,q;k-1},\\
 \partial_K^*(\overline{Z}^* \wedge \omega_j^{p,q;k}) =\ &2ij(k-(p+q)+j) I^* \wedge \overline{Z}^* \wedge \omega_{j-1}^{p-1,q;k-1} \\
 &- 2i(p-j)(n-k+q-j+1) I^* \wedge\overline{Z}^* \wedge \omega_j^{p-1,q;k-1},
\end{align*} 
 as well as
\begin{align*}
 \partial_K^*(I^* \wedge \omega_j^{p,q;k}) =\ &2ij(n+1-(p+q) + j) I^* \wedge \overline{Z}^* \wedge \omega_{j-1}^{p-1,q-1;k-1} \\
 &-2i(p-j)(q-j)I^* \wedge \overline{Z}^* \wedge  \omega_j^{p-1,q-1;k-1}, \\
 \partial_K^*(I^* \wedge Z^* \wedge \overline{Z}^* \wedge \omega_{j}^{p,q;k}) =\ &2(k-n+1)I^* \wedge \overline{Z}^* \wedge \omega_j^{p,q;k}.
 \end{align*}
\end{prop}

\begin{proof}
  In order to compute the image of
  $\omega := Z^* \wedge \overline{Z}^* \wedge \omega_j^{p,q;k}$ under $\partial_K^*$
  we have to determine the image of $\ast_K\omega$ under the partial derivative
  $\overline{\partial}_K$. Recall that for an $M$-invariant linear functional
  $\alpha$ on $(\frak g/\frak m)_{\mathbb C}^m$, the exterior derivative of the
  corresponding invariant differential form is induced by the functional $d\alpha$
  which sends vectors $X_0 + \frak m, \dotso, X_{m} + \frak m \in \frak g/\frak m$ to
\begin{align}\label{eq:exterior_derivative_appendix}
 \textstyle\sum_{i < j} (-1)^{i+j} \alpha([X_i, X_j] + \frak m, X_0 + \frak m, \dotso, \hat{\imath}, \dotso, \hat{\jmath}, \dotso, X_{m} + \frak m),
\end{align}
where the hats denote omission. Furthermore, if $\alpha$ is of $K$-type $(r,s)$ and
bidegree $(r+s, \ell)$, $\overline{\partial}_K\alpha$ is computed using
\eqref{eq:exterior_derivative_appendix} by inserting $\ell$ vectors in
$(\mathfrak{k}/\mathfrak{m})_{\mathbb{C}}$ and $r+s+1$ vectors in
$(\mathfrak{p}/\mathfrak{m})_{\mathbb{C}}$ of which $r$ are holomorphic and $s+1$ are
antiholomorphic.

In our case the element $\alpha = \ast_K\omega$ is trivial upon insertion of any
invariant vector $Z$, $\overline{Z}$ and $I$. Furthermore, any representative of
these vectors in $\mathfrak{g}$ is contained in the direct sum of all even grading
components, whereas for $X$, $Y \in \mathbb{C}^n$ the representatives $\xi_X$ of
$F_X$ and $\eta_Y$ of $G_Y$ with trivial $\mathfrak{m}$-part are contained in the odd
grading components. Therefore, the vector $[X_i, X_j] + \mathfrak{m}$ is
$M$-invariant unless it is induced by a pairing of an invariant vector and one of the
vectors $\xi_X$ or $\eta_Y$. In addition, using the bracket relations on
$\mathfrak{g}$ the pairing with $Z$ preserves holomorphic vectors and antiholomorphic
vectors, hence by the $K$-type of $\ast_K \omega$ the corresponding summands in
\eqref{eq:exterior_derivative_appendix} are trivial. All in all we deduce that
\begin{align}\label{eq:K_codifferential_omega}
 \overline{\partial}_K\ast_K\omega = \overline{Z}^* \wedge \iota_{\overline{Z}}\overline{\partial}_K \ast_K\omega + I^* \wedge \iota_I\overline{\partial}_K \ast_K\omega.
\end{align}

For the first summand in \eqref{eq:K_codifferential_omega}, let us denote the representative of $\overline{Z}$ in $\frak g$ with trivial $\frak m$-part by the same symbol. Then the pairings of this vector with the holomorphic and antiholomorphic parts of $\xi_X$ and $\eta_Y$ are given by
\begin{align*}
 [\overline{Z}, \xi_X^{1,0}] &= \frac{1}{2}\xi_X^{1,0}, &  [\overline{Z}, \xi_X^{0,1}] &= \frac{1}{2} \xi_X^{0,1}, & [\overline{Z}, \eta_Y^{1,0}]&= 0, & [\overline{Z}, \eta_Y^{0,1}] &= 2 \xi_Y^{0,1} - \frac{1}{2}  \eta_Y^{0,1}.
 \end{align*}
Therefore, since $\ast_K\omega$ is of bidegree $(2n-(p+q), 2k-(p+q))$ a direct computation using \eqref{eq:exterior_derivative_appendix} shows that 
\begin{align}\label{eq:proof_proposition_a3_derivative_omega_Z}
 \iota_{\overline{Z}} \overline{\partial}_K\ast_K\omega = (k-n) \ast_K\omega.
\end{align}

In order to compute the other summand of \eqref{eq:K_codifferential_omega}, we write $\omega$ as the wedge product of $\omega_{0,2}^{k-(p+q)+j}$ and $\tilde{\omega} := Z^* \wedge \overline{Z}^* \wedge \omega_{j}^{p,q;p+q-j}$, which does not contain $\omega_{0,2}$ as a factor. Since both $\overline{\partial}_K$ and $\iota_I$ are antiderivations and $\iota_I \ast_K\omega = 0$  we get
\begin{align}\label{eq:proof_proposition_a3_derivative_omega_I}
 \iota_I\overline{\partial}_K\ast_K\omega = \left(\iota_I\overline{\partial}_K\ast_K\tilde{\omega}\right) \wedge \omega_{0,2}^{k-(p+q)+j}  + \left(\ast_K\tilde{\omega}\right) \wedge \iota_I\overline{\partial}_K\omega_{0,2}^{k-(p+q)+j}.
\end{align}
From Section \ref{3.4} we know that $\iota_I\overline{\partial}_K\omega_{0,2} = - \overline{\omega_{1,1}}$, which we insert into the above expression and afterwards apply Theorem \ref{thm_wedge_product_2_form_omega}(ii) to rewrite $\overline{\omega_{1,1}} \wedge \ast_K\tilde{\omega}$ as an image of the $K$-Hodge star operator.

Thus, it remains to compute $\iota_I \overline{\partial}_K \ast_K \tilde{\omega}$,
which will be done using formula \eqref{eq:exterior_derivative_appendix}. Let us
denote the representative of $I$ in $\frak g$ with trivial $\frak m$-part by the same
symbol. Regarding the degree of $\ast_K \tilde{\omega}$ a moment of thought shows
that the only nontrivial contribution in \eqref{eq:exterior_derivative_appendix} is
given by the summands containing the pairing of $I$ with $\xi_X$ for
$X \in \mathbb{C}^n$. Now a direct computation shows that
 \begin{align*}
 [I, \xi_X^{1,0}] &= i(\xi_X^{1,0} - \eta_X^{1,0}), & [I, \xi_X^{0,1}] &= -i(\xi_X^{0,1} - \eta_X^{0,1}),
\end{align*}
for all $X$, $Y \in \Bbb C^n$, so since the pairing with $I$ respects holomorphic and antiholomorphic vectors the only nontrivial pairing in \eqref{eq:exterior_derivative_appendix} is between $I$ and $\xi_X^{0,1}$. Therefore, using relation
\begin{align*}
 \iota_{G_Y^{0,1}}\ast_K\tilde{\omega} = (-1)^{p+q} 2(p-j)\iota_{F_Y^{0,1}}\left(Z^* \wedge \overline{Z}^* \wedge \ast_K\omega_{j}^{p-1,q;p+q-j-1}\right),
\end{align*}
which follows from Lemma \ref{lem_power_omega_2_0}, a direct computation yields
\begin{align*}
 \iota_I\overline{\partial}_K\!\ast_K\!\tilde{\omega} = (-1)^{p+q+1}2i(p-j)(n-q+1)\ast_K\!\left( Z^* \wedge \overline{Z}^* \wedge \omega_{j}^{p-1,q;p+q-j-1}\right).
\end{align*}
Inserting this back into \eqref{eq:proof_proposition_a3_derivative_omega_I} and move the power of $\omega_{0,2}$ back into the $K$-Hodge star we obtain by a direct computation that
\begin{align}\label{eq:proof_proposition_a3_derivative_omega_I_full}
\begin{aligned}
 \iota_I\overline{\partial}_K\ast_K\omega =\ &(-1)^{p+q}2i (p-j)(k-n-1-p+j) \ast_K(Z^* \wedge \overline{Z}^* \wedge \omega_j^{p-1,q;k-1}) \\
 &+ (-1)^{p+q} 2ij (k-(p+q)+j) \ast_K(Z^* \wedge \overline{Z}^* \wedge \omega_{j-1}^{p-1,q;k-1}).
 \end{aligned}
\end{align}

Finally, we insert \eqref{eq:proof_proposition_a3_derivative_omega_Z} and
\eqref{eq:proof_proposition_a3_derivative_omega_I_full} into
\eqref{eq:K_codifferential_omega} and apply the negative of the $K$-Hodge star to the
resulting equation. Then a direct computation yields the claimed formula.

The formulae for the images of the other Poisson kernels under $\partial_K^*$ follow
by combining the above with the formulae for $K$-derivative of the Poisson kernels of
low degree from \ref{3.4} as well as the adjointness of the wedge product and the
interior product with respect to the $K$-Hodge star. As an example, we compute the
$K$-codifferential of $\iota_Z\omega = \overline{Z}^* \wedge \omega_j^{p,q;k}$. In
order to do so, we first apply the $K$-Hodge star to this element, which by the
relation $2 Z^\flat = \overline{Z}^*$ coincides with
$\frac{1}{2}(-1)^{p+q} \overline{Z}^* \wedge \ast_K\omega$. Next, we apply the
operator $\overline{\partial}_K$, where we use that
$\overline{\partial}_K\overline{Z}^* = 0$ (c.f. Section \ref{3.4}). All in all,
 \begin{align*}
  \partial^*_K\left(\overline{Z}^* \wedge \omega_j^{p,q;k}\right) &= \frac{1}{2}(-1)^{p+q} \ast_K \left(\overline{Z}^* \wedge \overline{\partial}_K\ast_K\omega \right) = - \iota_Z \partial_K^*\omega,
 \end{align*}
where we used again the adjointness of the interior and the wedge product for the last equality.
\end{proof}
For the rest of this section we derive formulae for the adjoint of the $K$-Lefschetz
map. In order to do so, let $\omega_K$ be the pullback of the K\"{a}hler form on
$G/K$ and as before denote the corresponding $M$-invariant element in
$\Lambda^{2,0}(\frak g / \frak m)^*$ by the same symbol. Recall from Section
\ref{3.2} the definitions of the $K$-Lefschetz map $\Cal L_K$ and its adjoint
$\Cal L_K^*$. Since these maps are $G$-equivariant, they induce $M$-equivariant
linear maps on the level of the underlying representations, which we denote by the
same symbols.

\begin{prop}\label{prop_adjoint_K_Lefschetz}
Let $p, q$ and $k$ be integers so that $0 \le p, q, k-p, k-q \le n$. Then
\begin{align*}
 \Cal L_K^*\omega_j^{p,q;k} = 2j(n+1-(p+q)+j) \omega_{j-1}^{p-1,q-1;k-1} - 2(p-j)(q-j) \omega_j^{p-1,q-1;k-1}.
\end{align*}
Furthermore, we have
\begin{align*}
 \Cal L_K^*(Z^* \wedge \omega_j^{p,q;k}) &= Z^* \wedge \Cal L_K^*\omega_j^{p,q;k} & \Cal L_K^*(\overline{Z}^* \wedge \omega_j^{p,q;k}) &= \overline{Z}^* \wedge \Cal L_K^*\omega_j^{p,q;k} \\
  \Cal L_K^*(I^* \wedge \omega_j^{p,q;k}) &= I^* \wedge \Cal L_K^*\omega_j^{p,q;k}
  \end{align*}
  as well as
  \begin{align*}
    \Cal L_K^*(I^* \wedge Z^* \wedge \overline{Z}^* \wedge \omega_j^{p,q;k}) &= I^* \wedge Z^* \wedge \overline{Z}^* \wedge \Cal L_K^*\omega_j^{p,q;k} - 2i I^* \wedge \omega_j^{p,q;k}.
\end{align*}
\end{prop}
\begin{proof} Recall that the pullback of the K\"{a}hler form is
  $\omega_K = \frac{1}{2}(i Z^* \wedge \overline{Z}^* + \omega_{2,0})$. Inserting
  this into the formula for $\Cal L_K^*\alpha$, where
  $\alpha \in \Lambda^*(\frak g/\frak m)^*_{\mathbb{C}}$ is $M$-invariant, and using
  the relation
  $\left(Z^* \wedge \overline{Z}^* \wedge \ast_K\alpha\right) =
  4\ast_K\iota_Z\iota_{\overline{Z}} \alpha$ we deduce
\begin{align}\label{eq:colefschetz_map}
 \Cal L_K^*\alpha = 2i \iota_Z\iota_{\overline{Z}}\alpha + \frac{1}{2} \ast_K^{-1}(\omega_{2,0}\wedge \ast_K\alpha).
\end{align}
In the case $\alpha = \omega_j^{p,q;k}$ the invariant vectors $Z$ and $\overline{Z}$ insert trivially into $\alpha$, so the first summand of \eqref{eq:colefschetz_map} vanishes. For the second summand of \eqref{eq:colefschetz_map} we can apply Theorem \ref{thm_wedge_product_2_form_omega}(ii), which yields 
\begin{align*}
  \Cal L_K^*\omega_j^{p,q;k} = 2j(n+1-(p+q)+j)\omega_{j-1}^{p-1,q-1;k-1} - 2(p-j)(q-j)\omega_j^{p-1,q-1;k-1}.
\end{align*}
The expressions for the other Poisson kernels follow analogously.
\end{proof}

\end{document}